\documentclass[12pt, leqno]{article}

\usepackage{amsmath,amsthm}
\usepackage{amssymb,latexsym}
\usepackage{amscd,amstext}
\usepackage{amsfonts}
\usepackage{latexsym}
\usepackage{t1enc}
\usepackage[latin1]{inputenc}
\usepackage[english]{babel}
\usepackage{bm}

\usepackage{graphicx}
\usepackage{cite}
\usepackage{color}

\newtheorem{theorem}{Theorem}[section] 
\newtheorem{lemma}[theorem]{Lemma}
\newtheorem{proposition}[theorem]{Proposition} 
\newtheorem{corollary}[theorem]{Corollary} 
\newtheorem{conjecture}[theorem]{Conjecture} 

\theoremstyle{definition}
\newtheorem{definition}[theorem]{Definition} 
\newtheorem{example}[theorem]{Example}

\theoremstyle{remark}
\newtheorem{remark}[theorem]{Remark} 
\newtheorem{prep}[theorem]{Discussion}

\newcommand{\ignore}[1]{}

\allowdisplaybreaks[1]

\newcommand{\beq}{\begin{eqnarray}}
\newcommand{\eeq}{\end{eqnarray}}
\newcommand{\beqq}{\begin{eqnarray*}}
\newcommand{\eeqq}{\end{eqnarray*}}
\newcommand{\bit}{\begin{itemize}}
\newcommand{\eit}{\end{itemize}}
\newcommand{\ds}{\displaystyle}
\newcommand{\R}{\ensuremath{\mathbb R}}\newcommand{\real}{\R}
\newcommand{\Rn}{\ensuremath{{\mathbb R}^{n}}} \newcommand{\rn}{\Rn}
\newcommand{\N}{{\ensuremath{\mathbb N}}}\newcommand{\nat}{\N}
\newcommand{\No}{\ensuremath{\N_{0}}}\newcommand{\no}{\N_0}
\newcommand{\Z}{\mathbb Z} 
\newcommand{\Zn}{\Z^{n}} \newcommand{\zn}{\Zn}

\newcommand{\Dd}{\mathrm{D}}
\newcommand{\dint}{\;\mathrm{d}}

\newcommand{\B}{\ensuremath{B^s_{p,q}}}

\newcommand{\F}{\ensuremath{F^s_{p,q}}}

\newcommand{\supp}{\ensuremath{\mathrm{supp}\,}}
\newcommand{\bpr}{\begin{proof}}
\newcommand{\epr}{\end{proof}}
\newcommand{\bli}{\begin{list}{}{\labelwidth6mm\leftmargin8mm}}
\newcommand{\eli}{\end{list}}

\newcommand{\non}{\ensuremath{\mathbb{N}^n_0}}
\newcommand{\Non}{\non} 

\newcommand{\comp}{\ensuremath{\mathbb{C}}}

\newcommand{\HH}{\ensuremath{\mathbb H}}

\newcommand{\sul}{\underline{\mathfrak{s}}}
\newcommand{\sol}{\overline{\mathfrak{s}}}
\newcommand{\upind}[1]{\sol\left(\bm{#1}\right)}
\newcommand{\lowind}[1]{\sul\left(\bm{#1}\right)}
\newcommand{\dc}[1]{{#1}_{\Gamma}}
\newcommand{\paren}[1]{\bm{(}#1\bm{)}}
\newcommand{\whole}[1]{\ensuremath\left\lfloor #1 \right\rfloor} 
\newcommand{\wGamma}{\ensuremath{w_{\varkappa, \Gamma}}}
\newcommand{\dist}{{\mathbin{\mathrm{dist}\,}}}
\newcommand{\DRn}{\ensuremath{\mathcal{D}(\rn)}}
\newcommand{\DG}{\ensuremath{\mathcal{D}(\rn\setminus\Gamma)}}
\newcommand{\dicho}{\ensuremath{{\mathbb D}}}

\newcommand{\tr}[1]{\mathrm{tr}\rule[-0.5ex]{0ex}{1.2ex}_{#1}}

\newcommand{\Bsi}{B^{\bm{\sigma}}_{p,q}}
\newcommand{\Btau}{B^{\bm{\tau}}_{p,q}}

\newcommand{\Bsg}{{\mathbb{B}}^{\bm{\sigma}}_{p,q}(\Gamma)}

\newcommand{\Bsih}{B^{\bm{\sigma} \bm{h}^{1/p}\paren{n}^{1/p}}_{p,q}\!(\rn)}

\allowdisplaybreaks[1]
\numberwithin{equation}{section}

\begin{document}

\title{Traces for Besov spaces on fractal $h$-sets and dichotomy results}

\author{Ant{\'o}nio M. Caetano \\  Center for R\&D in Mathematics and Applications,\\ Department of Mathematics, University of Aveiro \\ 3810-193 Aveiro, Portugal \\ Email: acaetano@ua.pt
\and
Dorothee D. Haroske\\ Institute of Mathematics, Friedrich-Schiller-University Jena\\
07737 Jena, Germany\\ Email: dorothee.haroske@uni-jena.de
}

\date{}

\maketitle

\renewcommand{\thefootnote}{}



\footnote{\copyright 2015. Licensed under the CC-BY-NC-ND 4.0 license http://creativecommons.org/licenses/by-nc-nd/4.0/.}

\footnote{Formal publication: http://dx.doi.org/10.4064/sm8171-1-2016.}

\renewcommand{\thefootnote}{\arabic{footnote}}
\setcounter{footnote}{0}


\begin{abstract}
We study the existence of traces of Besov spaces on fractal $h$-sets $\Gamma$ with the special focus laid on necessary assumptions implying this existence, or, in other words, present criteria for the non-existence of traces. In that sense our paper can be regarded as an extension of \cite{bricchi-3} and a continuation of the recent paper \cite{Cae-env-h}. Closely connected with the problem of existence of traces is the notion of {\em dichotomy} in function spaces: We can prove that -- depending on the function space and the set $\Gamma$ -- there occurs an alternative: {\em either} the trace on $\Gamma$ exists, {\em or} smooth functions compactly supported outside $\Gamma$ are dense in the space. This notion was introduced by Triebel in \cite{T-dicho} for the special case of $d$-sets. 
\end{abstract}

\section{Introduction}

The paper is devoted to a detailed study of traces of regular distributions taken on fractal sets. 
Such questions are of particular interest in view of boundary value problems
of elliptic operators, where the solutions belong to some appropriate Besov (or Sobolev) space. One standard method is to start with assertions about traces on hyperplanes and then to transfer these
results to bounded domains with sufficiently smooth boundary
afterwards. Further studies may concern compactness or regularity results,
leading to the investigation of spectral properties. However, when it comes to irregular (or fractal) boundaries, one has to circumvent a lot of difficulties following that way, such that another method turned out to be more appropriate. This approach was proposed by Edmunds and Triebel in connection with smooth boundaries first in \cite{ET} and then extended to fractal $d$-sets in \cite{ET7,ET6,T-Frac}. Later the setting of $d$-sets was extended to $(d,\Psi)$-sets by Moura \cite{moura-diss} and finally to the more general $h$-sets by Bricchi \cite{bricchi-diss}. \\
The idea is rather simple to describe, but the details are much more complicated: at first one determines the trace spaces of certain Besov (or Sobolev) spaces as precisely as possible, studies (compact) embeddings of such spaces into appropriate target spaces together with their entropy and approximation numbers afterwards, and finally applies Carl's or Weyl's inequalities to link eigenvalues and entropy or approximation numbers. If one is in the lucky situation that, on one hand, one has atomic or wavelet decomposition results for the corresponding spaces, and, on the other hand, the irregularity of the fractal can be characterised by its local behaviour (within `small' cubes or balls), then there is some chance to shift all the arguments to appropriate sequence spaces which are usually easier to handle. This is one reason for us to stick to fractal $h$-sets and Besov spaces at the moment. But still the problem is not so simple and little is known so far. Dealing with spac!
 es on $h$-sets we refer to \cite{CL-2,CL-3,K-Z,lopes-diss}, and, probably closest to our approach here, 
 \cite[Chapter 8]{T-F3}. There it turns out that one first needs a sound knowledge about the existence and quality of the corresponding trace spaces. Returning to the first results in that respect in \cite{bricchi-diss}, see also \cite{bricchi-fsdona,bricchi-3,bricchi-4}, we found that the approach can (and should) be extended for later applications. More precisely, for a positive continuous and non-decreasing
function $h : (0,1]\to\real$ (a gauge function) with $\lim_{r\to 0} h(r)=0$, a non-empty compact set $\Gamma\subset\rn$ is called {\em $h$-set} if there exists a finite Radon measure $\mu$ in $\Rn$ with $\supp\mu=\Gamma$ and 
\[
\mu(B(\gamma,r)) \sim h(r),\qquad r\in(0,1],\ \gamma\in\Gamma,
\]
see  also \cite[Chapter~2]{Rogers} and \cite[p.~60]{mattila}. In the special case $h(r)=r^d$, $0< d< n$, $\Gamma$ is called $d$-set (in the sense of \cite[Def.~3.1]{T-Frac}, see also \cite{JW-1,mattila} -- be aware that this is different from \cite{falconer}). Recall that some self-similar fractals are outstanding examples of $d$-sets; for instance,
the usual (middle-third) Cantor set in $\real^1$ is a $d$-set for $d = \ln 
2/\ln 3$, and the Koch curve in $\real^2$ is a $d$-set for $d=\ln 4/\ln
3$. \\
The trace is defined by completion of pointwise traces of $\varphi\in\mathcal{S}(\rn)$, assuming that 
for $0<p<\infty$ we have in addition $\ds \| \varphi\raisebox{-0.5ex}[1ex][1ex]{$\big|_{\Gamma}$}~ | L_p(\Gamma) \| \lesssim \  \| \varphi | B^t_{p,q}(\rn)\|$ for suitable parameters $t\in\real$ and $0<q<\infty$. 
In case of a compact $d$-set $\Gamma$, $0<d<n$, this results in
\begin{equation}
\tr{\Gamma} B^{\frac{n-d}{p}}_{p,q}(\rn) = L_p(\Gamma)\quad \text{if} \quad 0<q\leq\min\{p,1\}
\label{Lp-trace-d}
\end{equation}
and for $s>\frac{n-d}{p}$,
\[
\tr{\Gamma}\B(\rn)  = {\mathbb B}^{s-\frac{n-d}{p}}_{p,q}(\Gamma),
\]
see \cite{T-Frac} with some later additions in \cite{T-func,T-F3}. Here $\B(\rn)$ are the usual Besov spaces defined on $\rn$. In the classical case $d=n-1$, $0<p<\infty$, $0<q\leq\min\{p,1\}$ this reproduces the well-known trace result $\tr{\real^{n-1}} B^{\frac1p}_{p,q}(\rn) = L_p(\real^{n-1})$. \\
In case of $h$-sets $\Gamma$ one needs to consider Besov spaces of generalised smoothness $\Bsi(\rn)$ which naturally extend $\B(\rn)$: instead of the smoothness parameter $s\in\real$ one now admits sequences $\bm{\sigma} = (\sigma_j)_{j\in\no}$ of positive numbers which satisfy $ \sigma_j  \sim  \sigma_{j+1}$, $j\in\no$. Such spaces are special cases of $B^{\bm{\sigma},N}_{p,q}(\rn)$ studied in \cite{F-L} recently, but they are known for a long time: apart from the interpolation approach 
(with a function parameter), see  \cite{merucci,C-F}, there is the rather abstract approach
(approximation by series of entire analytic functions and coverings) developed independently by Gol'dman and Kalyabin in the late 70's
and early 80's of the last century; we refer to the
survey \cite{Kal-Liz} and the appendix \cite{lizorkin} which cover the
extensive (Russian) literature at that time. We shall rely on the Fourier-analytical
approach as presented in \cite{F-L}. It turns out that the classical smoothness $s\in\real$ has to be replaced by certain regularity indices $\upind{\sigma}$, $\lowind{\sigma}$ of $\bm{\sigma}$. In case of $\bm{\sigma}= (2^{js})_j$ the spaces $\Bsi(\rn)$ and $\B(\rn)$ coincide and $\upind{\sigma}=\lowind{\sigma}=s$  in that case. Dealing with traces on $h$-sets $\Gamma$ in a similar way as for $d$-sets, one obtains
\[
\tr{\Gamma} \Btau(\rn)= \Bsg,
\]
where the sequence $\bm{\tau}$ (representing smoothness) depends on $\bm{\sigma}$, $h$ (representing the geometry of $\Gamma$) and the underlying $\rn$; in particular, with $\bm{h} :=
\left(h(2^{-j})\right)_j$, $ \bm{h_{p}} = \left(h(2^{-j})^\frac1p\
2^{j\frac{n}{p}}\right)_j$, the counterpart of \eqref{Lp-trace-d} reads as
\[
\tr{\Gamma} B^{\bm{h_{p}}}_{p,q}(\rn) = L_p(\Gamma), \qquad 0<p<\infty, \quad 0<q\leq\min\{p,1\}. 
\]
These results were already obtained in \cite{bricchi-3} under some additional restrictions. In \cite{Cae-env-h} we studied sufficient conditions for the existence of such traces again (in the course of dealing with growth envelopes, characterising some singularity behaviour) and return to the subject now to obtain `necessary' conditions, or, more precisely, conditions for the non-existence of traces. This problem is closely connected with the so-called {\em dichotomy}: Triebel coined this notion in \cite{T-dicho} for, roughly speaking, the following alternative: the existence of a trace on $\Gamma$ (by completion of pointwise traces) on the one hand, and the density of the set of smooth functions compactly supported outside $\Gamma$, denoted by $\mathcal{D}(\rn\setminus \Gamma)$, on the other hand. Though it is rather obvious that the density of $\mathcal{D}(\rn\setminus \Gamma)$ in some space prevents the existence of a properly defined trace, it is not clear (and, in fact,!
  not true in general) that there is some close connection vice versa. However, in some cases there appears an alternative that {\em either} we have an affirmative answer to the density question {\em or} traces exist. The criterion which case occurs naturally depends on the function spaces and the set $\Gamma$. Our main outcome in this respect, Theorem~\ref{coroleither}, establishes the following: if the $h$-set $\Gamma$ satisfies, in addition, some porosity condition, $\bm{\sigma}$ is an admissible sequence and either $1\leq p<\infty$, $0<q<\infty$, or $0<q\leq p<1$, then
\begin{align*}
\text{\upshape\bfseries either} &\quad \Bsg = \tr{\Gamma} B^{\bm{\sigma} \bm{h_{p}}}_{p,q}\!(\rn)\quad \mbox{exists}\\
\text{\upshape\bfseries or} &\quad  \mathcal{D}(\rn\setminus \Gamma) \quad \mbox{is dense in}\quad
B^{\bm{\sigma} \bm{h_{p}}}_{p,q}\!(\rn)\\
& \quad \mbox{and, therefore,}\quad \tr{\Gamma}B^{\bm{\sigma} \bm{h_{p}}}_{p,q}\!(\rn)\ \mbox{cannot exist}.
\end{align*}
This result is later reformulated in terms of the dichotomy introduced in \cite{T-dicho}. Note that there are further related approaches to trace and dichotomy questions  in \cite{CS-4,CS-trace-dom} for Besov spaces defined by differences, and in \cite{iwona-2,Ha-EE} referring to weighted settings.

The paper is organised as follows. In Section~\ref{sect-1} we collect some fundamentals about $h$-sets and  Besov spaces of generalised smoothness, including their atomic decomposition. In Section~\ref{sect-2} we  turn to trace questions with our main result being Theorem~\ref{coroleither}, before we finally deal with the dichotomy and obtain Corollary~\ref{Coro-2.28}. Throughout the paper we add remarks, discussions and examples to illustrate the (sometimes technically involved) arguments and results.

\section{Preliminaries}
\label{sect-1}
%
\subsection{General notation}
%
%

As usual, $\rn$ denotes the $n$-dimensional real Euclidean space,
$\nat$ the collection of all natural numbers and $\nat_0=\nat\cup
\{0\}$. We use the equivalence `$\sim$' in
$$
a_k \sim b_k \quad \mbox{or} \quad \varphi(x) \sim \psi(x)
$$
always to mean that there are two positive numbers $c_1$ and $c_2$
such that
$$
c_1\,a_k \leq b_k \leq c_2\, a_k \quad \mbox{or} \quad
c_1\,\varphi(x) \leq \psi(x) \leq c_2\,\varphi(x)
$$
for all admitted values of the discrete variable $k$ or the
continuous variable $x$, where $(a_k)_k$, $(b_k)_k$ are
non-negative sequences and $\varphi$, $\psi$ are non-negative
functions. If only one of the inequalities above is meant, we use the symbol $\, \lesssim \,$ instead.
Given two quasi-Banach spaces $X$ and $Y$, we write $X
\hookrightarrow Y$ if $X\subset Y$ and the natural embedding of
$X$ into $Y$ is continuous.\\
All unimportant positive constants will be denoted by $c$,
occasionally with additional subscripts within the same formula.
If not otherwise indicated, $\log$ is always taken with respect to
base 2. For some $\varkappa\in\real$ let 
\begin{equation}
\varkappa_+ = \max\{\varkappa , 0\}\quad\mbox{and}\quad \whole{\varkappa} =
\max\{ k\in\mathbb{Z} : k\leq \varkappa\} \, .   
\label{a+}
\end{equation}
Moreover, for $0<r\leq\infty$
the number $r'$ is given by $\, \frac{1}{r'}:=\left(1-\frac1r\right)_+ $ . 

For convenience,      
let both $\ \dint x\ $ and $\ |\cdot|\ $ stand for the
($n$-dimensional) Lebesgue measure in the sequel. The notation $\ |\cdot|\ $
is also used for the size of an $n$-tuple in $\Non$ and the Euclidean norm in $\Rn$, while $\ |\cdot|_\infty \ $ is
reserved for the corresponding infinity norm.

Given $x\in\rn$ and $r>0$, $B(x,r)$ denotes the closed ball
\begin{equation}
B(x,r) = \left\{y\in\rn : |y-x|\leq r\right\}.
\end{equation} 

%
\subsection{$h$-sets $\Gamma$}

A central concept for us is the theory of so-called $h$-sets and corresponding
measures, we refer to a comprehensive treatment of this concept in
\cite{Rogers}. Certainly one of the most prominent sub-class of these sets are
the famous $d$-sets, see also Example~\ref{ex-h-fnt} below, but it is also
well-known that in many cases more general approaches are necessary,
cf. \cite[p.~60]{mattila}. Here we essentially follow the presentation in
\cite{bricchi-diss,bricchi-fsdona,bricchi-3,bricchi-4}, see also
\cite{mattila} for basic notions and concepts.

\begin{definition}
\bli
\item[{\upshape\bfseries (i)\hfill}]
Let $\HH$ denote the class of all positive continuous and non-decreasing
functions $h : (0,1]\to\real$ {\em (gauge functions)} with $\lim\limits_{r\to 0} h(r)=0$. 
\item[{\upshape\bfseries (ii)\hfill}]
Let $h\in\HH$. A non-empty compact set $\Gamma\subset\rn$ is called
{\em $h$-set} if there exists a finite Radon measure $\mu$ in $\Rn$ with
\begin{align}
\supp\mu & =\Gamma,   \label{supp-h}   \\
\mu(B(\gamma,r))& \sim h(r),\qquad r\in(0,1],\ \gamma\in\Gamma.
\label{hset}
\end{align}
If for a given $h\in\HH$ there exists an $h$-set $\Gamma\subset\rn$, we call
    $h$ a {\em measure function (in~\rn)} and any related measure $\mu$
    with \eqref{supp-h} and~\eqref{hset} will be called {\em $h$-measure (related to $\Gamma$)}.
\eli
\label{defi-hset}
\end{definition}

We quote some results on $h$-sets and give examples afterwards; we refer to
the above-mentioned books and papers for proofs and a more detailed account on
geometric properties of $h$-sets. \\

In view of (ii) the question arises which $h\in\HH$ are measure functions. We
give a necessary condition first, see \cite[Thm.~1.7.6]{bricchi-diss}.

\begin{proposition}     \label{h-measfnt}
Let $h\in\HH$ be a measure function. Then there exists some $c>0$ such that for all $j,k\in\no$,
\begin{equation}    \label{main2}
\frac{h(2^{-k-j})}{h(2^{-j})}\ \geq \ c\ 2^{-kn}.
\end{equation}
\end{proposition}

\begin{remark}\label{rem-doubling}
Note that every $h$-set $\Gamma$ satisfies the {\em doubling condition}, i.e. 
there is some $c>0$ such that
\begin{equation}
\mu(B(\gamma,2r))  \ \leq \ c\ \mu(B(\gamma,r)),\qquad r\in(0,1],\ \gamma\in\Gamma.
\label{doubling}
\end{equation}
Obviously one can regard \eqref{main2} as a refined version of \eqref{doubling}
for the function~$h$, in which the 
dimension~$n$ of the underlying space \rn~ is taken into account (as expected). \\
The complete characterisation for functions
$h\in\HH$ to be measure functions is given in \cite{bricchi-fsdona}: 
There is a compact set $\Gamma$ and a Radon measure $\mu$ with \eqref{supp-h}
and \eqref{hset}
if, and only if, there are constants $0<c_1\leq c_2<\infty$ and a function $h^\ast \in \HH$ such that
\[
c_1 h^\ast(t) \leq h(t) \leq c_2 h^\ast(t),\quad t\in (0,1],
\]
and
\begin{equation}    \label{main2a}
h^\ast (2^{-j}) \leq 2^{kn} h^\ast(2^{-k-j}), \qquad \text{for all}\quad j,k\in\no.
\end{equation}
\end{remark}

\begin{proposition} \label{T:h-sets}
Let $\Gamma$ be  an $h$-set in \rn. 
All $h$-measures
$\mu$ related to $\Gamma$ are equivalent to ${\mathcal H}^h|\Gamma$,
where the latter stands for the restriction to~$\Gamma$ of the generalised
Hausdorff measure with respect to the gauge function~$h$.
\end{proposition}

\begin{remark}\label{rem-dim}
A proof of this result is given in \cite[Thm.~1.7.6]{bricchi-diss}. 
Concerning the theory of generalised 
Hausdorff measures ${\mathcal H}^h$ we refer to \cite[Chapter~2]{Rogers} and
\cite[p.~60]{mattila}; in particular, 
if $h(r)=r^d$, then ${\mathcal H}^h$ coincides 
with the usual $d$-dimensional Hausdorff measure.
\end{remark}

We recall a description of measure functions and explicate a
few examples afterwards.

\begin{proposition}~Let $n\in\nat$.\quad\\[-4ex]
\bli
\item[{\upshape\bfseries (i)\hfill}]
Let $\xi:(0,1]\to[0,n]$ be a measurable function. Then the function
\begin{equation}\label{h-xi}
h(r)=\exp\Bigl\{-\int_r^1\xi(s)\,\frac{\dint s}{s}\Bigr\},\qquad r\in(0,1]
\end{equation}
is a measure function.
\item[{\upshape\bfseries (ii)\hfill}]
Conversely, let $h$ be a given measure function. Then for any $\varepsilon>0$
there exists a measurable function $\xi:(0,1]\to [-\varepsilon, n+\varepsilon]$
  such that
\begin{equation}\label{h-xi-2}
h(r) \sim \exp\Bigl\{-\int_r^1\xi(s)\,\frac{\dint s}{s}\Bigr\},\qquad r\in(0,1].
\end{equation}
\eli
\label{measure-xi}
\end{proposition}

This version of the theorem is given in \cite[Thm.~3.7]{bricchi-3}; it can
also be identified as a special case of a result in \cite[pp.~74]{BGT}.

\begin{example}\label{ex-h-fnt}
We restrict ourselves to a few examples only, but in view of Proposition~\ref{measure-xi} one can easily find 
further examples, we refer to \cite[Ex.~3.8]{bricchi-3}, too. All functions
are defined for $r\in (0,\varepsilon)$, suitably extended on $(0,1]$ afterwards.

Let $\Psi$ be a continuous {\em admissible} function or a continuous {\em
  slowly varying} function, respectively. An
{\em admissible} function $\Psi$ in the sense of \cite{ET6,moura-diss} is a
positive monotone function on $(0,1]$ such that $\Psi\left(2^{-2j}\right) \sim
\Psi\left(2^{-j}\right)$, $j\in\nat$. A positive and measurable function
  $\Psi$ defined on the interval $(0,1]$ is said to be {\em slowly varying}
  (in Karamata's sense) if
\begin{equation} \label{a.f.equiv}
\lim_{t\rightarrow 0}\frac{\Psi(s t)}{\Psi(t)}=1,\quad s\in(0,1].
\end{equation}
For such functions it is known, for instance, that for any $\delta >0$ there exists
$c=c(\delta)>1$ such that $ \frac{1}{c}\,s^{\delta} \leq
 \frac{\Psi(st)}{\Psi(t)} \leq c\, s^{-\delta}$, for $ t,s\in
 (0,1]$, and for each $\varepsilon>0$ there is a non-increasing
 function $\phi$ and a non-decreasing  function $\varphi$
with $t^{-\varepsilon}\,\Psi(t)\sim \phi(t)$, and
$t^{\varepsilon}\,\Psi(t)\sim \varphi(t)$; 
we refer to the monograph \cite{BGT} for details and further properties; see also
\cite[Ch.~V]{zygmund}, \cite{EKP}, and \cite{neves-01,neves-02}.
In particular,
\begin{equation}
\Psi_b(x)=\left(1+|\log x|\right)^b, \quad x\in (0,1],\quad
b\in\real, \label{psi-ex-log}
\end{equation}
may be considered as a prototype both for an admissible function and a slowly
varying function. \\
Let $0<d<n$. Then 
\begin{equation}
\label{hPsi}
h(r)=r^d\ \Psi(r),\quad r\in (0,1],
\end{equation}
is a typical example for $h\in\HH$. The limiting cases $d=0$ and $d=n$ can be
included, assuming additional properties of $\Psi$ in view of  \eqref{main2}
and $h(r)\to 0$ for $r\to 0$, e.g.
\begin{equation}
h(r)=(1+|\log r|)^b, \quad b<0,\quad r\in (0,1],\label{h-log-ex}
\end{equation}
referring to \eqref{psi-ex-log}. 

Later on we shall need to consider measure functions of this kind when $b \in [-1,0)$, so we would like to mention here that with this restriction the proof that the above is a measure function is a simple application of the characterisation given just before Proposition \ref{T:h-sets} (take $h^\ast = h$ and observe that, for $b \in [-1,0)$, $(1+x)^b 2^{xn} \geq 1$ for $x=0$ and for $x \geq 1$, from which follows (\ref{main2a}) taking $x=k$).

Such functions $h$ given by \eqref{hPsi} are related
to so-called $(d,\Psi)$-sets studied in \cite{ET6,moura-diss}, whereas
the special setting $\Psi\equiv 1$ leads to 
\begin{equation}
\label{h-d}
h(r)=r^d,\quad r\in (0,1],\quad 0< d< n,
\end{equation}
connected with the famous $d$-sets. Apart from \eqref{hPsi} also functions of
type $h(r)=\exp\left(b |\log r|^\varkappa\right)$, $b<0$, $0<\varkappa<1$, are admitted.
\end{example}

We shall need another feature of $h$-sets, namely the so-called `porosity'
condition, see also \cite[p.156]{mattila} and \cite[Sects.~9.16-9.19]{T-func}.

\begin{definition}
\label{porosity}
A Borel set $\Gamma\neq\emptyset$ satisfies the {\em porosity condition} if
there exists a number $0<\eta<1$, such that for any ball $B(\gamma,r)$
centered at $\gamma\in\Gamma$ and with radius $0<r\leq 1$, there is a ball
$B(x,\eta r)$ centered at some $x\in\rn$ satisfying
\begin{equation}
B(\gamma,r) \ \supset \ B(x,\eta r), \quad B(x,\eta r) \ \cap \
\overline{\Gamma} \ = \ \emptyset.
\label{ball-1}
\end{equation}
\end{definition}

Replacing $\eta$ by $\frac{\eta}{2}$, we can complement \eqref{ball-1} by
\begin{equation}
\mathrm{dist}\ \left(B(x,\eta r), \overline{\Gamma}\right) \ \geq \ \eta r,
\quad 0<r\leq 1.
\label{ball-2}
\end{equation}
This definition coincides with \cite[Def.~18.10]{T-Frac}. There is a complete
characterisation for measure functions $h$ such that the corresponding
$h$-sets $\Gamma$ satisfy the porosity condition; this can be found in
\cite[Prop.~9.18]{T-func}. We recall it for convenience.

\begin{proposition}
Let $\Gamma\subset\rn$ be an $h$-set. Then $\Gamma$ satisfies the porosity
condition if, and only if, there exist constants $c>0$ and $\varepsilon>0$
such that 
\begin{equation}
\frac{h\left(2^{-j-k}\right)}{h\left(2^{-j}\right)} \ \geq \ c\
2^{-(n-\varepsilon)k}, \quad j,k\in\no.
\label{ball-3}
\end{equation}
\label{crit-poros}
\end{proposition}

Note that an $h$-set $\Gamma$ satisfying the porosity condition has Lebesgue
measure $|\Gamma|=0$, but the converse is not true. This can be seen from
\eqref{ball-3} and the result \cite[Prop.~1.153]{T-F3}, 
\begin{equation}
|\Gamma| = 0 \quad\mbox{if, and only if,}\quad \lim_{r\to 0} \ \frac{r^n}{h(r)}
 = 0
\end{equation}
for all $h$-sets $\Gamma$.

\begin{remark}\label{example-ball}
In view of our above examples and \eqref{ball-3} it is obvious that $h$ from
\eqref{hPsi} and \eqref{h-d} with $d=n$ does not satisfy the porosity condition, unlike in
case of $d<n$.
\end{remark}

Let $L_p(\Omega)$, $\Omega\subseteq\rn$, $0<p\leq \infty$, stand for the usual
quasi-Banach space of $p$-integrable  (measurable, essentially bounded if
$p=\infty$) functions with respect to the Lebesgue measure, quasi-normed by 
$$
\|f\mid L_p(\Omega)\|:= \Bigl(\int_{\Omega} |f(x)|^p \dint x \Bigr)^{\!1/p},
$$
with the obvious modification if $p=\infty$. Moreover, when $\Gamma\subset\rn$
is an $h$-set in the sense of Definition~\ref{defi-hset}, then we consider
$L_p(\Gamma)=L_p(\Gamma,\mu)$ as the usual quasi-Banach space of $p$-integrable  (measurable, essentially bounded if $p=\infty$) functions on $\Gamma$ with respect to the measure $\mu$, quasi-normed by 
$$
\|f\mid L_p(\Gamma)\| = \Bigl(\int_{\Gamma} |f(\gamma)|^p \mu(\dint \gamma)
\Bigr)^{\!1/p} < \infty 
$$
for $0<p<\infty$, and
$$ \|f\mid L_\infty(\Gamma)\| = \inf\left\{ s>0~: \ \mu\left(\{\gamma\in\Gamma
: |f(\gamma)|>s\}\right)=0\right\} < \infty.
$$
In view of Proposition~\ref{T:h-sets} all (possibly different) measures $\mu$
corresponding to $h$ result in the same $L_p(\Gamma)$ space.

\subsection{Function spaces of generalised smoothness}
%
We deal with the concept of admissible sequences first.
\begin{definition}
A sequence $\bm{\sigma} = (\sigma_j)_{j\in\no}$ of positive numbers is called {\em
  admissible} if there are two positive constants $\ d_0$, $d_1$ such that
\begin{equation}
d_0 \ \sigma_j \ \leq \ \sigma_{j+1}\ \leq \ d_1\ \sigma_j, \quad j\in\no.
\label{adm-seq}
\end{equation}
\label{defi-sigma-adm}
\end{definition}

\begin{remark}\label{R-adm-1}
If $\bm{\sigma}$ and $\bm{\tau}$ are admissible sequences, then $\bm{\sigma
  \tau} := \left( \sigma_j \tau_j\right)_j\ $ and $\bm{\sigma}^r :=
\left(\sigma_j^r\right)_j\ $, $r\in\real$, are admissible, too. For later use
we introduce the  notation
\begin{equation}
\paren{a} := \left(2^{ja}\right)_{j\in\no}\quad \mbox{where}\quad a\in\real,
\end{equation}
that is $\paren{a}=\bm{\sigma}$ with $\sigma_j = 2^{ja}$, $j\in\no$. 
Obviously, for $a,b\in\real$, $r>0$, and $\bm{\sigma}$ admissible, we have 
$\paren{a} \paren{b} = \paren{a+b}$,
$\paren{\textstyle\frac{a}{r}}=\paren{a}^{1/r}$, and $\paren{a}\bm{\sigma} = 
\left(2^{ja}\sigma_j \right)_{j\in\no}$. 
\end{remark}

\begin{example}\label{ex-sigma}
Let $ s\in\real$, $\Psi $ be an admissible function in the sense of
Example~\ref{ex-h-fnt} above. Then $ \bm{\sigma} = 
\left(2^{js} \Psi\left(2^{-j}\right)\right)_j $ is admissible, including, in
particular, $ \bm{\sigma} = \paren{s}$, $ s\in\real$. We refer to
\cite{F-L} for a more general approach and further examples.
\end{example}

We introduce some `regularity' indices for $\bm{\sigma}$. 

\begin{definition}
Let $\bm{\sigma}$ be an admissible sequence, and
\begin{equation}
\label{s_sigma}
\lowind{\sigma} := \liminf_{j\to\infty} \
\log\left(\frac{\sigma_{j+1}}{\sigma_j}\right) 
\end{equation}
and
\begin{equation}
\label{s^sigma}
\upind{\sigma} := \limsup_{j\to\infty} \
\log\left(\frac{\sigma_{j+1}}{\sigma_j}\right). 
\end{equation}
\end{definition}

\begin{remark}\label{R-Boyd}
These indices were introduced and used in \cite{bricchi-diss}. For admissible sequences $\bm{\sigma}$
according to \eqref{adm-seq} we have $\ \log d_0 \leq 
 \lowind{\sigma} \leq \upind{\sigma} \leq \log d_1$. One easily
verifies that
\begin{equation}
 \upind{\sigma} = \lowind{\sigma} = s\quad \mbox{in case of}\quad  \bm{\sigma} =
 \left(2^{js} \Psi\left(2^{-j}\right)\right)_j 
\label{index-ex}
\end{equation}
for all admissible functions $\Psi $ and $s\in\real$. On the other hand one
can find examples in \cite{F-L}, due to 
\textsc{Kalyabin}, showing that an admissible sequence has not
necessarily a fixed main order. Moreover, it is known that 
for any $0<a\leq b<\infty$, there is
an admissible sequence $\bm{\sigma}$ with
$\lowind{\sigma} =a$ and $\upind{\sigma} =b$,
that is, with prescribed upper and lower indices.\\

For later use we fix some observations that are more or less immediate consequences of the definitions
\eqref{s_sigma}, \eqref{s^sigma}. Let $\bm{\sigma}, \bm{\tau}$ be 
admissible sequences. Then
\begin{equation}
\label{index-1}
\lowind{\sigma} = -
\sol\left(\bm{\sigma}^{-1}\right), \qquad 
\sul\left(\bm{\sigma}^r\right) = r\
\lowind{\sigma}, \quad r \geq 0, 
\end{equation}
and
\begin{equation}
\upind{\sigma \tau} \leq \upind{\sigma} + \upind{\tau},\qquad 
\lowind{\sigma \tau} \geq \lowind{\sigma} + \lowind{\tau}.\label{index-2}
\end{equation}
In particular, for $\bm{\sigma} = \paren{a}$, $a\in\real$, then
\eqref{index-2} can be sharpened by
\begin{equation}
\sol\left(\bm{\tau} \paren{a}\right)   = a + \upind{\tau}, \quad 
\sul\left(\bm{\tau} \paren{a}\right) = a + \lowind{\tau}. 
\label{index-3}
\end{equation}
Observe that,
given $\varepsilon>0$, there are two positive constants
$c_1=c_1(\varepsilon)$ and $c_2=c_2(\varepsilon)$ such that
\begin{equation} \label{estimate}
c_1 \ 2^{(\lowind{\sigma}-\varepsilon)j} \leq \sigma_j \leq c_2 \
2^{(\upind{\sigma}+\varepsilon)j},\qquad j\in\nat_0.
\end{equation}
Plainly this implies that whenever $\lowind{\sigma} >0$, then
$\bm{\sigma}^{-1}$ belongs to any space $\ell_u$, $0<u\leq\infty$, whereas
$\upind{\sigma}<0 $ leads to $\bm{\sigma}^{-1}\not\in\ell_\infty$.
\end{remark}

\begin{remark}\label{R-boyd-2}
Note that in some later papers, cf. \cite{bricchi-4}, instead of \eqref{s_sigma} and
  \eqref{s^sigma} the so-called {\em upper} and {\em lower Boyd
  indices of} $\bm{\sigma}$ are considered, given by 
\begin{equation}
\alpha_{\bm{\sigma}} = \lim_{j\to\infty} \ \frac1j \ 
\log \left(\sup_{k\in\no} \ \frac{\sigma_{j+k}}{\sigma_k}\right) \ = \ \inf_{j\in\nat} \ \frac1j \ 
\log \left(\sup_{k\in\no} \ \frac{\sigma_{j+k}}{\sigma_k}\right)
\end{equation}
and
\begin{equation}
\beta_{\bm{\sigma}} = \lim_{j\to\infty} \ \frac1j \ 
\log \left(\inf_{k\in\no} \ \frac{\sigma_{j+k}}{\sigma_k}\right) \ = \ \sup_{j\in\nat} \ \frac1j \ 
\log \left(\inf_{k\in\no} \ \frac{\sigma_{j+k}}{\sigma_k}\right)
\end{equation}
respectively. In general we have 
$$
\lowind{\sigma}  \leq \beta_{\bm{\sigma}} \leq  \alpha_{\bm{\sigma}}\leq
 \upind{\sigma},
$$
but one can construct admissible sequences with $\lowind{\sigma}  <
\beta_{\bm{\sigma}}$ and $\alpha_{\bm{\sigma}} <  \upind{\sigma}$. 
\end{remark}

We want to introduce function spaces of generalised smoothness and need
to recall some notation. By $\mathcal{S}(\rn)$ we denote the Schwartz
space of all complex-valued, infinitely differentiable and rapidly
decreasing functions on $\rn$ and by $\mathcal{S}'(\rn)$ the dual space
of all tempered distributions on $\rn$. If $\varphi \in
\mathcal{S}(\rn)$, then
\begin{equation}\label{Ftransform}
\widehat{\varphi}(\xi)\equiv (\mathcal F
\varphi)(\xi):=(2\pi)^{-n/2}\int_{\rn}e^{-ix\xi}\varphi(x) \dint
x, \quad \xi\in \rn,
\end{equation}
denotes the Fourier transform of $\varphi$. As usual, ${\mathcal
F}^{-1} \varphi$ or $\varphi^{\vee}$ stands for the inverse
Fourier transform, given by the right-hand side of
\eqref{Ftransform} with $i$ in place of $-i$. Here $x\xi$ denotes
the scalar product in $\rn$. Both $\mathcal F$ and ${\mathcal
F}^{-1}$ are extended to $\mathcal{S}'(\rn)$ in the standard way. Let
$\varphi_0\in\mathcal{S}(\rn)$ be such that
\begin{equation}  \label{phi}
\varphi_0(x)=1 \quad \mbox{if}\quad |x|\leq 1 \quad \mbox{and}
\quad \supp \varphi_0 \subset \{x\in\rn: |x|\leq 2\},
\end{equation}
and for each $j\in\nat$ let
\begin{equation}\label{phi-j}
\varphi_j(x):=\varphi_0(2^{-j}x)-\varphi_0(2^{-j+1}x), \quad x\in
\rn.
\end{equation}
Then the sequence $(\varphi_j)_{j=0}^{\infty}$ forms a smooth dyadic resolution of
unity.

\begin{definition}
Let $\bm{\sigma}$ be an admissible sequence, $0<p,q\leq\infty$, and
$(\varphi_j)_{j=0}^{\infty}$ a smooth dyadic resolution of
unity (in the sense described above). Then
\begin{equation}
\Bsi(\rn) = \left\{ f\in \mathcal{S}'(\rn): \biggl( \sum_{j=0}^\infty
\sigma_j^q \left\| {\mathcal F}^{-1} \varphi_j {\mathcal F}f |
L_p(\rn)\right\|^q \biggr)^{1/q} < \infty\right\}
\end{equation}
$($with the usual modification if $q=\infty)$.
\end{definition}
\begin{remark}\label{R-Bsigma}
These spaces are quasi-Banach spaces, independent of the chosen resolution of
unity, and $\mathcal{S}(\rn)$ is dense in $\Bsi(\rn)$ when $p<\infty$ and
$q<\infty$. Taking $\bm{\sigma} = \left(2^{js}\right)_j\ $, $ 
s\in\real$, we obtain the classical Besov spaces $\B(\rn)$, whereas $\ \bm{\sigma}
= \left(2^{js} \Psi\left(2^{-j}\right)\right)_j\ $, $ s\in\real$, $\Psi $ an
admissible function, leads to spaces $B^{(s,\Psi)}_{p,q}(\rn)$, studied in
\cite{moura,moura-diss} in detail. Moreover, the above spaces $\Bsi(\rn)$
are special cases of the more general approach investigated in \cite{F-L}. For
the theory of spaces $\B(\rn)$ we refer to the series of monographs 
\cite{T-F1,T-F2,T-Frac,T-func,T-F3}.
\end{remark}

We recall the atomic characterisation of spaces $\Bsi(\rn)$, for later use. 
Let $\zn$ stand for the lattice of all points in $\rn$
with integer-valued components, $Q_{j m}$ denote a cube in
$\rn$ with sides parallel to the axes of coordinates, centered at
$2^{-j}m=(2^{-j}m_1,\dots,2^{-j}m_n)$, and with side length
$2^{-j}$, where $m\in \zn$ and $j \in \no$.
If $Q$ is a cube in $\rn$ with sides parallel to the axes of coordinates and $r>0$, then $rQ$ is the cube in $\rn$
concentric with $Q$, with sides parallel to the sides of $Q$ and $r$ times the length.

\begin{definition} \label{atoms-Bsigma}
Let $K \in \nat_0$ and $b> 1$.\\[-4ex]
\bli
\item[{\upshape\bfseries (i)\hfill}]
A $K$ times differentiable
complex-valued function $a(x)$ in $\rn$ $($continuous if $K=0)$ is
called {\em an $1_K$-atom} if
\begin{equation}\label{atom-supp-0}
\supp a \subset b\,Q_{0 m},\quad\mbox{for some}\quad m\in\zn,
\end{equation}
and
$$
|\Dd^{\alpha}a(x)|\leq 1, \quad \mbox{for} \quad |\alpha|\leq K.
$$
\item[{\upshape\bfseries (ii)\hfill}]
Let $L+1\in\nat_0$, and $\bm{\sigma}$ be admissible. A $K$
times differentiable complex-valued function $a(x)$ in $\rn$
$($continuous if $K=0)$ is called {\em an $(\bm{\sigma}, p)_{K,L}$-atom}
if for some $j \in \no$,
\begin{align}
\supp a \subset b\,Q_{j m},\quad & \mbox{for some}\quad m\in \zn,
\label{atom-supp}\\
|\Dd^{\alpha}a(x)|\leq \sigma_j^{-1} \ 2^{j(\frac{n}{p}+|\alpha|)}\, \quad
& \mbox{for} \quad |\alpha|\leq K,\quad x\in\rn,\label{atom-deri}
\intertext{and}
\int_{\rn} x^{\beta}a(x) \dint x=0, \quad & \mbox{if}
\quad|\beta|\leq L.
\label{atom-mom}
\end{align}
\eli
\end{definition}

We adopt the usual convention to denote atoms located at $Q_{j m}$ (which means \eqref{atom-supp-0} or \eqref{atom-supp}) by
$a_{jm}$, $j\in\no$, $m\in\zn$. For  sequences $\lambda =
\left(\lambda_{jm}\right)_{j\in\no, \ m\in\zn}$ of complex numbers the
Besov sequence spaces $b_{p,q}$, $0<p,q\leq\infty$, are given by
\begin{equation}
\lambda \in b_{p,q} \quad\mbox{if, and only if,}\quad
\left\|\lambda| b_{p,q}\right\| = \left(\sum_{j=0}^\infty \left(\sum_{m\in\zn}
\left|\lambda_{jm}\right|^p\right)^{q/p}\right)^{1/q} < \infty
\end{equation}
(with the usual modification if $p=\infty$ or $q=\infty$). 
The atomic decomposition theorem for $\Bsi(\rn)$ reads as follows, see
\cite[Thm.~4.4.3, Rem. 4.4.8]{F-L}.

\begin{proposition} \label{atomicdecomposition}
Let $\bm{\sigma}$ be admissible, $b>1$, $0<p,q\leq\infty$, $K\in \no$ and $L+1\in\no$ with
\begin{equation}
K > \upind{\sigma}, \qquad\mbox{and}\qquad L > -1 +
n\left(\frac1p -1\right)_+ - \lowind{\sigma}
\label{moment-Bsi}
\end{equation}
be fixed. Then $f\in {\mathcal S}'(\rn)$ belongs to
$\Bsi(\rn)$ if, and only if, it can be represented as
\begin{equation}\label{series}
f=\sum_{j=0}^{\infty} \sum_{m\in\zn} \lambda_{j m}\,a_{j m}(x), 
\quad\mbox{convergence being in } {\mathcal S}'(\rn),
\end{equation}
where $\lambda \in b_{p,q}$ and $a_{j m}(x)$ are $1_K$-atoms $(j=0)$ or
$(\bm{\sigma},p)_{K,L}$-atoms $(j \in \nat)$ according to 
Definition~\ref{atoms-Bsigma}. Furthermore,
$$
\inf \|\lambda\,|\,b_{p,q}\|,
$$
where the infimum is taken over all admissible representations
\eqref{series}, is an equivalent quasi-norm in $\Bsi(\rn)$.
\end{proposition}

\begin{remark}
\label{tessellations}
For later use it is useful to remark that, for fixed $j\in \No$, the family
$$
\{ Q_{jm} : \, m \in \Zn \}
$$
constitutes a tessellation of $\Rn$. We shall call it a {\em tessellation of step} $2^{-j}$ and denote each cube of it by a (corresponding) {\em grid cube}.
\end{remark}

\begin{remark}
\label{partitions}
The following shall also be of use later on: given the families of tessellations as above (for any possible $j \in \No$), clearly there exists, for each $j \in \No$, a partition of unity 
$$
\{ \varphi_{jm} : \, m \in \Zn \}
$$
of $\Rn$ by functions $\varphi_{jm}$ supported on $\frac{3}{2} Q_{jm}$ and such that, for each $\gamma \in \Non$, there exists $c_\gamma > 0$ independent of $j$ such that
\begin{equation}
|\Dd^\gamma \varphi_{jm}(x)|\leq c_\gamma 2^{j|\gamma|}, \quad x \in \Rn,\; m \in \Zn.
\label{derivcontrol}
\end{equation}
We shall call it a {\em partition of unity of step} $2^{-j}$.
\end{remark}

\section{Besov spaces on $\Gamma$}
\label{sect-2}

Let $\Gamma$ be some $h$-set, $h\in \HH$. We follow \cite{bricchi-diss} and use the abbreviation
\begin{equation}
\bm{h} := \left(h_j\right)_{j\in\no} \quad \mbox{with}\quad  h_j := h(2^{-j}),
\quad j\in\no,
\label{h-seq}
\end{equation}
for the sequence connected with $h\in\HH$.

\subsection{Trace spaces ${\mathbb B}^{\bm{\sigma}}_{p,q}(\Gamma)$}
%
\label{sect-2-1}

Recall that $L_p(\Gamma) = L_p(\Gamma, \mu)$, where $\mu \sim {\mathcal
  H}^h|\Gamma$ is related to the $h$-set $\Gamma$. Suppose there exists
some $c>0$ such that for all $\varphi\in\mathcal{S}(\rn)$,
\begin{equation}
\label{trace-def}
\left\| \varphi\raisebox{-0.4ex}[1.1ex][0.9ex]{$|_{\Gamma}$}~ | L_p(\Gamma)
\right\| \leq \ c\ \left\| \varphi | \Btau(\rn)\right\|,
\end{equation}
where the restriction on $\Gamma$ is taken pointwisely. By the density of
$\mathcal{S}(\rn)$ in $\Btau(\rn)$ for $p,q<\infty$ and the completeness of
$L_p(\Gamma)$  one can thus define for $f\in \Btau(\rn)$ its
trace $\tr{\Gamma} f $ on
$\Gamma$ by completion of pointwise restrictions. \\ 

This was the approach followed in \cite[Def. 3.3.5]{bricchi-diss}, conjugating general embedding results for Besov spaces on $\Rn$ (as seen for example in \cite[Thm. 3.7]{C-Fa} or \cite[Prop. 2.2.16]{bricchi-diss}) with the fact that \eqref{trace-def} above holds for $\bm{\tau}=\bm{h}^{1/p}\paren{n}^{1/p}$, $0<p<\infty$, $0<q\leq \min\{p,1\}$ (cf. \cite[Thm.~3.3.1(i)]{bricchi-diss}. Then we have, following \cite[Def. 3.3.5]{bricchi-diss}, that for $0<p,q<\infty$ and $\bm{\sigma}$ admissible with  $\lowind{\sigma}
> 0$,  Besov spaces $\Bsg$ on $\Gamma$ are defined as

\begin{equation} 
\label{defBesovGamma}
\Bsg := \tr{\Gamma} \Bsih,
\end{equation}
more precisely,
\begin{equation} 
\label{precisely}
\Bsg := \left\{ f\in L_p(\Gamma) : \exists \ g\in \Bsih,\tr{\Gamma} g =   
    f\right\},
\end{equation}
equipped with the quasi-norm
\begin{equation}
\label{qnormBG}
\left\| f |  \Bsg\right\| = \inf\left\{ \left\| g | \Bsih\right\| :  
   \tr{\Gamma} g =  f,  g\in \Bsih\right\}.
\end{equation}

This was extended in the following way in \cite[Def. 2.7]{Cae-env-h}:

\begin{definition}
Let $0<p,q<\infty$, $\bm{\sigma}$ be an admissible sequence, and $\Gamma$ be an
$h$-set. Assume that \\[-4ex]
\bli
\item[{\upshape\bfseries (i)\hfill}]
in case of $p\geq 1$ or $q\leq p<1$, 
\begin{equation}
\label{ext-2-4-a}
\bm{\sigma}^{-1} \in \ell_{q'},
\end{equation}
\item[{\upshape\bfseries (ii)\hfill}]
in case of $0<p<1$ and $p<q$,
\begin{equation}
\bm{\sigma}^{-1} \bm{h}^{\frac1r-\frac1p}\in \ell_{v_r} \quad \text{for some} \   
r\in [p,\min\{q,1\}] ~ \mbox{and} \ 
\frac{1}{v_r} = \frac1r-\frac1q, 
\label{ext-2-4-b}
\end{equation}
\eli
is satisfied. Then (it makes sense to) define $\Bsg$  as in \eqref{defBesovGamma}, \eqref{precisely}, again with the  quasi-norm given by \eqref{qnormBG}.
\label{defi-Bsg}
\end{definition}

\begin{remark}
The reasonability of declaring that smoothness $\paren{1}$  (that is, 0, in classical notation) on $\Gamma$ corresponds to smoothness $\bm{h}^{1/p}\paren{n}^{1/p}$ on $\Rn$ --- as is implicit in \eqref{defBesovGamma} --- comes from the fact that, at least when $\Gamma$ also satisfies the porosity condition, we have indeed that $\tr{\Gamma} B_{p,q}^{\bm{h}^{1/p}\paren{n}^{1/p}}(\Rn) = L_p(\Gamma)$ when $0<p<\infty$, $0<q\leq \min\{p,1\}$  --- cf. \cite[Thm. 3.3.1(ii)]{bricchi-diss}.
\end{remark}

\begin{remark}
Both in \cite{bricchi-diss} and in \cite{Cae-env-h} the definition of Besov spaces on $\Gamma$ also covers the cases when $p$ or $q$ can be $\infty$, following some convenient modifications of the approach given above. However, in view of the main results to be presented in this paper, the restriction to $0<p,q<\infty$ is natural for us here, as will be apparent in the following subsection.
\end{remark}

\begin{remark}\label{R-ext-def}
We briefly compare the different assumptions in \cite[Def.~3.3.5]{bricchi-diss} and in
Definition~\ref{defi-Bsg} above. Due to the observation made immediately after \eqref{estimate}, $\lowind{\sigma}
> 0$ implies $\bm{\sigma}^{-1} \in \ell_v$ for arbitrary $v$,
i.e. \eqref{ext-2-4-a} and \eqref{ext-2-4-b} with $r=p$. The converse,
however, is not true, take e.g. $\sigma_j = 
(1+j)^\varkappa$, $\varkappa\in\real$, then $\lowind{\sigma} =
0$, but $\bm{\sigma}^{-1}\in \ell_{q'}  $ for $\ \varkappa >
\frac{1}{q'}$, corresponding to \eqref{ext-2-4-a}. As for
\eqref{ext-2-4-b}, say with $p=r$, for the same $\bm{\sigma}$ given above we have also that
$\bm{\sigma}^{-1}\in\ell_{v_p}$ for $\varkappa > \frac1p-\frac1q$, but
still $\lowind{\sigma}=0$. So the above definition is in fact a proper extension of the one considered in \cite{bricchi-diss} and, as we shall see, will (at least in some cases) be indeed the largest possible extension.
\end{remark}

\begin{remark}\label{R-09-3}
The definition above applies in particular when $\Gamma$ is a $d$-set and $\bm{\sigma} = \paren{s}$ with
$s>0$. In simpler notation, we can write in this case that 
\begin{equation} 
\mathbb{B}^s_{p,q}(\Gamma) = \tr{\Gamma} B^{s+\frac{n-d}{p}}_{p,q}(\rn).
\end{equation}
This coincides with \cite[Def.~20.2]{T-Frac}.
\end{remark}


\subsection{Criteria for non-existence of trace spaces}\label{sect-2-2}
Here we shall try to get necessary conditions for the existence of the trace. To this end, we explore the point of view that the trace cannot exist in the sense used in Definition~\ref{defi-Bsg} when $\mathcal{D}(\rn\setminus\Gamma)$, the set of test functions with compact support outside $\Gamma$, is dense in $\Bsih$.

In fact, assume that  $\mathcal{D}(\rn\setminus\Gamma)$ is dense in $\Btau(\Rn)$.
Let $\varphi\in C^\infty_0(\rn)$ with
$\varphi\equiv 1$ on a neighbourhood of $\Gamma$. Clearly, $\varphi\in
\Btau(\rn)$.  Then there exists a sequence $(\psi_k)_k \subset
\mathcal{D}(\rn\setminus\Gamma)$ with
\begin{equation}
\left\| \varphi - \psi_k | \Btau(\rn) \right\|
\xrightarrow[k\to\infty]{} 0.
\end{equation}
If the trace were to exist in the sense explained before, this would imply 
\begin{equation}
0 = \psi_k\raisebox{-0.4ex}[1ex][0.9ex]{$|_{\Gamma}$} = \tr{\Gamma}
\psi_k \xrightarrow[k\to\infty]{} \tr{\Gamma} \varphi =
\varphi\raisebox{-0.4ex}[1ex][0.9ex]{$|_{\Gamma}$} = 1 \quad
\text{in}\ L_p(\Gamma),
\end{equation}
which is a contradiction.

\begin{prep} \label{prep-dense}
So in order to disprove the existence of the
trace in certain cases it is sufficient to show the density of
$\mathcal{D}(\rn\setminus\Gamma)$ in $\Btau(\rn)$. Further, we may restrict ourselves to functions $\varphi\in
C^\infty_0(\rn)$ because of their density in $\Btau(\rn)$,
$0<p,q<\infty$, and approximate them by functions $\psi_k\in
\mathcal{D}(\rn\setminus\Gamma)$. We shall construct such  $\psi_k$
based on finite sums of the type
\[
\sum_{r\in I_k} \lambda_r \varphi_r,
\]
where $I_k$ is some finite index set, $\lambda_r\in\comp$ and $\varphi_r$
are compactly supported smooth functions with
\begin{equation}
\label{coincidence}
\sum_{r\in I_k} \lambda_r \varphi_r \varphi = \varphi
\end{equation}
on a neighbourhood (depending on $k$) of $\Gamma$, $\varphi$ as above, and
\begin{equation}
\left\| \sum_{r\in I_k} \lambda_r \varphi_r \varphi |
\Btau(\rn)\right\| \xrightarrow[k\to\infty]{} 0.
\label{tends0}
\end{equation}
Plainly, then $\psi_k := \varphi - \sum_{r\in I_k} \lambda_r
\varphi_r \varphi \in \mathcal{D}(\rn\setminus\Gamma)$, $k\in\nat$, 
\[
\left\| \varphi - \psi_k | \Btau(\rn)\right\|  = \left\| \sum_{r\in I_k} \lambda_r \varphi_r \varphi|
\Btau(\rn)\right\| \xrightarrow[k\to\infty]{} 0,
\]
and therefore the required density is proved.
\end{prep}

The limit in \eqref{tends0} above is going to be computed with the help of atomic representations for the functions considered. In order to explain how this will be done, we need first to consider a preliminary result which connects the Definition \ref{defi-hset} of $h$-sets with the structure of the atomic representation theorem given in Proposition \ref{atomicdecomposition}.

\begin{definition}
Given a set $\Gamma \subset \Rn$ and $r>0$, we shall denote by $\Gamma_r$ the neighbourhood of radius $r$ of $\Gamma$. That is,
$$\Gamma_r := \{ x \in \Rn : \, {\rm dist}(x,\Gamma) < r \}.$$
\end{definition}

\begin{lemma}
\label{optimcover}
Let $\Gamma$ be an $h$-set in $\Rn$ with $\mu$ a corresponding Radon measure. Let $j \in \N$. There is a cover $\{ Q(j,i) \}_{i=1}^{N_j}$ of $\, \Gamma_{2^{-j}}$ such that 
\begin{list}{}{\labelwidth7mm\leftmargin9mm}
\item[{\upshape\bfseries (i)\hfill}]
$N_j \sim h_j^{-1}$,
\item[{\upshape\bfseries (ii)\hfill}]
$Q(j,i)$ are cubes in $\Rn$ of side length $\, \sim 2^{-j}$,
\item[{\upshape\bfseries (iii)\hfill}]
each $Q(j,i)$ is, in fact, the union of $\, \sim 1$ grid cubes of a tessellation of $\Rn$ of step $2^{-j}$ in accordance with Remark \ref{tessellations},
\item[{\upshape\bfseries (iv)\hfill}]
each $Q(j,i)$ contains a ball of radius $\, \sim 2^{-j}$ centred at a point of $\, \Gamma$ such that any pair of them, for different values of $i$, are disjoint.
\eli
   
Denote by ${\cal Q}(j)$ the family of all grid cubes obtained in this way (that is, ${\cal Q}(j) = \{ Q_{jm} : Q_{jm} \subset Q(j,i) \mbox{ for some } i \}$) and consider the corresponding functions of a related partition of unity $\{ \varphi_{jm} \}$ of step $2^{-j}$ of $\Rn$ in accordance with Remark \ref{partitions}. That is, consider only the functions of that partition which are supported on $\frac{3}{2}Q_{jm}$ such that $Q_{jm} \in {\cal Q}(j)$.

The cover of $\Gamma_{2^{-j}}$ considered above can be chosen in such a way that the family 
$$
\{ \varphi_{jm} : \, Q_{jm} \in {\cal Q}(j) \}
$$
is a partition of unity of $\Gamma_{2^{-j}}$.

All equivalence constants in the estimates above are independent of $j$ and $i$.
\end{lemma}

\begin{proof}
For each $j \in \N$, start by considering an optimal cover of $\Gamma$ in the sense of \cite[Lemma 1.8.3]{bricchi-diss}, that is, a cover by balls $B(\gamma_i,2^{-j-1})$ centred at points $\gamma_i \in \Gamma$ and such that $B(\gamma_i,\frac{2^{-j-1}}{3}) \cap B(\gamma_k,\frac{2^{-j-1}}{3}) = \emptyset$ for $i \not= k$. As was pointed out in \cite[Lemma 1.8.3]{bricchi-diss}, the number $N_{2^{-j-1}}$ of balls in such a cover satisfy
\begin{equation}
N_{2^{-j-1}} \sim h_{j+1}^{-1}.
\label{coverBricchi}
\end{equation}
It is not difficult to see that each $B(\gamma_i,2^{-j-1})$ is contained in the union of $2^n$ grid cubes of the type $Q_{jm}$ which together form a cube of side length $2^{-j+1}$ and which we shall provisionally denote by $Q(j,i)$. Then assertions (ii)-(iv) of the lemma are clear. As to (i), it is also clear, using \eqref{coverBricchi} and Proposition \ref{h-measfnt}, that $N_j \lesssim h_j^{-1}$. The reverse inequality follows from the fact that there can only be $\, \sim 1$ different $B(\gamma_k,2^{-j-1})$ giving rise to the same $Q(j,i)$, because each one of those produces also a ball $B(\gamma_k,\frac{2^{-j-1}}{3})$ which is disjoint from all the other balls obtained in a similar way. Eliminating the repetitions among the previously considered $Q(j,i)$ and redefining the $i$ accordingly, we get a possible cover of $\Gamma$ satisfying (i)-(iv) of the lemma.

In order to get from here a corresponding cover of $\Gamma_{2^{-j}}$, we just need to enlarge each one of the previous $Q(j,i)$ by adding to it all the $4^n-2^n$ surrounding grid cubes which touch it. It is clear that the new cubes $Q(j,i)$ obtained in this way also satisfy (i)-(iv) above.

Finally, if we further enlarge the preceding $Q(j,i)$ by adding to it all the $6^n-4^n$ surrounding grid cubes which touch it, clearly we do not destroy properties (i)-(iv) for the new cubes $Q(j,i)$ obtained in this way, this cover of $\Gamma_{2^{-j}}$ satisfying also the property that 
$$
\{ \varphi_{jm} : \, Q_{jm} \in {\cal Q}(j) \}
$$
is a partition of unity of $\Gamma_{2^{-j}}$.
\end{proof}

\begin{remark}
\label{gridcover}
It follows from the construction above that ${\cal Q}(j)$ is also a cover of $\Gamma_{2^{-j}}$ with $\sharp{\cal Q}(j) \sim h_j^{-1}$ (though a property like (iv) above cannot be guaranteed for each grid cube in ${\cal Q}(j)$).
\end{remark}

\begin{remark}
\label{optimalcover}
We shall use the expression ``optimal cover of $\Gamma_{2^{-j}}$'' when referring to a cover of $\Gamma_{2^{-j}}$ satisfying all the requirements of the lemma above.
\end{remark}

We return now to the question of calculating the limit \eqref{tends0} in Discussion \ref{prep-dense}.

\begin{prep}
\label{prep-dense-2}
The index set $I_k$, $k\in\nat$, will have the structure
\begin{equation}
\label{Ik}
I_k = \left\{ (j,m)\in\nat\times\Zn: j \in J_k, \ m \in M_j \right\},
\end{equation}
where $J_k$ and $M_j$ are appropriately chosen finite subsets of $\N$ and $\Zn$ respectively. In any case, for each $(j,m) \in I_k$ the relation ${\rm dist}(Q_{jm},\Gamma) \lesssim 2^{-j}$ should be satisfied. With $r=(j,m)\in I_k$, we shall also require that 
\begin{equation}
\label{support}
{\rm supp\,} \varphi_r \subset b\, Q_{jm},
\end{equation}
for some constant $b>1$, and, for arbitrary fixed $K \in \N$,
\begin{equation}
\label{derivatives}
|\Dd^\gamma \varphi_r(x)|\leq c_K 2^{j|\gamma|}, \quad x\in\rn, \quad \gamma\in\non \, \mbox{ with } \, |\gamma| \leq K.
\end{equation} 
We claim that then (apart from constants) the functions $a_{(j,m)} = \tau_j^{-1}
2^{j\frac{n}{p}} \varphi_{(j,m)}\varphi$ are $(\bm{\tau},p)_{K,-1}$-atoms (no moment conditions) located at $Q_{jm}$: the support condition is obvious; as for the
derivatives we calculate for $\alpha\in\non$, $|\alpha|\leq K$, that
\begin{align*}
|\Dd^\alpha a_{(j,m)}(x)| & \leq \ \tau_j^{-1}
2^{j\frac{n}{p}} \sum_{\gamma\leq \alpha} {\alpha\choose\gamma}|\Dd^{\gamma}\varphi_{(j,m)}(x)| |\Dd^{\alpha-\gamma}\varphi(x)| \\
&\leq \ \tau_j^{-1} 2^{j\frac{n}{p}} c'_K 2^{j|\alpha|} \left\|
\varphi | C^K(\rn)\right\| \\
& \leq C_{K,\varphi} \tau_j^{-1}
2^{j(\frac{n}{p}+|\alpha|)},
\end{align*}
fitting the needs of Definition~\ref{atoms-Bsigma}.

Since 
\[
\sum_{r\in I_k} \lambda_r \varphi_r \varphi = \sum_{(j,m)\in I_k} (\lambda_{(j,m)} C_{K,\varphi} \tau_j 2^{-j\frac{n}{p}}) (C_{K,\varphi}^{-1} a_{(j,m)}) ,
\]
to the effect of obtaining \eqref{tends0} we can then estimate from above the quasi-norm of this function by applying Proposition \ref{atomicdecomposition}, in case the parameters considered do not require moment conditions.

If moment conditions are required, in the presence of porosity of the $h$-set $\Gamma$ a standard procedure can be applied in order to identify the above sum at least in a smaller neighbourhood of $\Gamma$ (smaller than the one considered apropos of \eqref{coincidence}) with an atomic representation

\[
\sum_{(j,m)\in I_k} (c_2^{-1} \lambda_{(j,m)} C_{K,\varphi} \tau_j 2^{-j\frac{n}{p}}) (c_2 C_{K,\varphi}^{-1} \tilde{a}_{(j,m)})
\]
with the appropriate moment conditions and, as is apparent, producing essentially the same upper estimate by application of Proposition \ref{atomicdecomposition}. So, in this case we replace, in Discussion \ref{prep-dense}, $\sum_{r \in I_k} \lambda_r \varphi_r \varphi$ by the above sum, keeping a property like \eqref{coincidence}. In any case, following Proposition  \ref{atomicdecomposition}, the quasi-norms of both sums in $\Btau(\Rn)$ are estimated by
\begin{equation}
\label{atomicuppest}
 \lesssim  \left( \sum_{j\in J_k} \left( \sum_{m \in M_j} | \lambda_{(j,m)} \tau_j 2^{-j\frac{n}{p}} |^p \right)^{q/p} \right)^{1/q}.
\end{equation}
\end{prep}

The standard procedure referred to above, leading to the creation of atoms with appropriate moment conditions, though still coinciding with the former atoms in a somewhat smaller neighbourhood of a set satisfying the porosity condition, comes from \cite{TW} and was, for example, used in \cite{Cae-00} and, more recently, in \cite{T-dicho}. It is quite technical, but we give here a brief description, for the convenience of the reader:

	For each $Q_{jm}$ as above, namely with ${\rm dist}(Q_{jm},\Gamma) \lesssim 2^{-j}$, fix an element of $\Gamma$, $x_{j,m}$ say, at a distance $\, \lesssim 2^{-j}$ of $Q_{jm}$. Clearly, there is a constant $c_1 > 0$ 
such that $Q_{jm} \subset B(x_{j,m},c_1 2^{-j})$. Without loss of generality, we can assume
that $0 < c_1 2^{-j} < 1$, so that, from the fact $\Gamma$ satisfies the porosity condition given in Definition \ref{porosity}, 
there exists $y_{j,m} \in \Rn$ such that $B(y_{j,m}, \eta c_1 2^{-j}) \subset B(x_{j,m}, c_1 2^{-j})$
and $B(y_{j,m}, \eta c_1 2^{-j}) \cap \Gamma = \emptyset $, where $0 < \eta <1$ is as in 
Definition \ref{porosity}. Obviously, as mentioned in \eqref{ball-2}, we can also say that ${\rm dist }(B(y_{j,m}, \eta c_1 2^{-j-1}) , \Gamma ) \geq \eta c_1 2^{-j-1}$.

	Fix a natural $L$ as in \eqref{moment-Bsi} --- with $\tau$ instead of $\sigma$ --- and let $\psi _ \gamma$, with $\gamma \in \N_0^n$
and $|\gamma | \leq L$, be $C^\infty $-functions $\psi _\gamma $ with support in the open ball ${\stackrel{\; \circ }{B}}(0,1)$ and satisfying the property
$$	\forall \beta ,\gamma \in \N_0^n \mbox{ with } |\beta |, |\gamma | \leq L, \: \int_{\Rn}{x^\beta 
\psi _\gamma (x)} \dint x = \delta_{\beta \gamma },$$

\noindent where $\delta_{\beta \gamma }$ stands for the Kronecker symbol (for a proof of the existence of such functions, see \cite[p. 665]{TW}).

Define, for each $j,m$ as above,
$$	d_{\gamma}^{j,m} \equiv \int_{\Rn}{x^{\gamma} a_{(j,m)} (\eta c_1 2^{-j-1} x + y_{j,m})} \dint x,
 \; \gamma \in \Non \mbox{ with } |\gamma | \leq L,$$

\noindent and
$$	\tilde{a}_{(j,m)}(z) = a_{(j,m)}(z) - \sum_{|\gamma| \leq L}{d_{\gamma }^{j,m} \psi _\gamma ( (\eta c_1
)^{-1} 2^{j+1} (z - y_{j,m}))}, \; z \in \Rn.$$

\noindent It is easy to see that
$$	\int_{\Rn} {( (\eta c_1)^{-1} 2^{j+1} (z - y_{j,m})) ^\beta \tilde{a}_{(j,m)} (z)} \dint z = 0, \; \beta \in \Non 
\mbox{ with } |\beta| \leq L,$$

\noindent and, consequently (by Newton's binomial formula),
$$	\int_{\Rn} {z^\beta \tilde{a}_{(j,m)} (z)} \dint z = 0, \; \beta \in \Non \mbox{ with } |\beta| \leq L.$$

\noindent This means that each $\tilde{a}_{(j,m)}$ has the required moment conditions for the atoms in the
atomic representations of functions of $\Btau(\Rn)$. Actually, it is not difficult to see that there exists a 
positive constant $c_2$ such that $c_2 C_{K,\varphi}^{-1}\tilde{a}_{(j,m)}$ is a $(\tau,p)_{K,L}$-atom located at $Q_{jm}$.

Since, from the 
hypotheses and choices that have been made, $\tilde{a}_{(j,m)} = a_{(j,m)}$ on $\Gamma_{\eta c_1 2^{-j-1}}$, we are through.

\begin{proposition}\label{dicholemma}
Let $\Gamma$ be an $h$-set satisfying the porosity condition, $0< p,q<\infty$ and 
$\bm{\tau}$ an admissible sequence.
Assume that
\begin{equation}
\bm{\tau}^{-1} \bm{h}^{1/p} \paren{n}^{1/p} \not\in \ell_\infty.
\label{dense-10}
\end{equation}
Then  $\mathcal{D}(\rn\setminus \Gamma)$ is dense in
$B^{\bm{\tau}}_{p,q}(\rn)$, therefore 
\begin{equation}
\tr{\Gamma} B^{\bm{\tau}}_{p,q}(\rn) \quad \mbox{cannot exist.}
\label{dense-9}
\end{equation}
\end{proposition}

\begin{proof}
By standard
reasoning we conclude from \eqref{dense-10} that there exists a
divergent subsequence of $\bm{\tau}^{-1} \bm{h}^{1/p} \paren{n}^{1/p}$. More
precisely, there is a strictly increasing sequence $(j_k)_{k\in\nat} \subset \nat$ such that
\begin{equation}
\lim_{k\to\infty} \ \tau_{j_k} \
  h_{j_k}^{-\frac1p} \ 2^{-j_k \frac{n}{p}} \ = \ 0.
\label{dense-3}
\end{equation}
For each $k\in \N$, consider an optimal cover of $\Gamma_{2^{-j_k}}$, in the sense of Remark \ref{optimalcover}, and  
follow our Discussions \ref{prep-dense} and \ref{prep-dense-2} with $I_k$ from \eqref{Ik} given by
$$
I_k = \{ (j,m) \in \N \times \Zn : \, j=j_k, \, Q_{jm} \in {\cal Q}(j) \}
$$
(that is, $J_k = \{ j_k \}$ and $M_j = M_{j_k} = \{ m \in \Zn : \,  Q_{j_km} \in {\cal Q}(j_k) \}$), $\varphi_{(j,m)} = \varphi_{jm}$ and $\lambda_{(j,m)} = 1$. Due to Remark \ref{partitions}, this fits nicely into Discussions \ref{prep-dense} and \ref{prep-dense-2}, in particular \eqref{coincidence}, \eqref{support} and \eqref{derivatives} hold. It then follows, specially from \eqref{atomicuppest}, that the quasi-norm of the sum in \eqref{tends0} (or of an alternative similar sum, as discussed in Discussion \ref{prep-dense-2} apropos of the consideration of moment conditions) is
\begin{align*}
\lesssim  \left( \sum_{Q_{j_km} \in {\cal Q}(j_k)} \tau_{j_k}^p
2^{-j_k n} \right)^{\frac1p} 
& = \ \tau_{j_k} \
2^{-j_k\frac{n}{p}}\ (\sharp {\cal Q}(j_k))^{\frac1p} \\ 
& \sim \ \tau_{j_k} \ 2^{-j_k\frac{n}{p}}\
h_{j_k}^{- \frac1p} \ \xrightarrow[k\to\infty]{} \ 0,
\end{align*}
where we have also used Remark \ref{gridcover} and \eqref{dense-3}. Thus \eqref{tends0} holds (possibly with an alternative similar sum if moment conditions are required, as mentioned above) and Discussion \ref{prep-dense} then concludes the proof.
\end{proof}

\begin{remark}
\label{noporosity}
The porosity condition was only used in the proof above to guarantee that atoms with appropriate moment conditions can be considered. Thus the proposition holds without assuming porosity of $\Gamma$ as long as $L$ can be taken equal to -1 in the atomic representation theorem for $\Btau(\Rn)$ (cf. Proposition \ref{atomicdecomposition} with $\bm{\tau}$ instead of $\bm{\sigma}$).
\end{remark}

The next result shows that when $p<q$ we can conclude as in the preceding proposition with an hypothesis weaker than \eqref{dense-10}.

\begin{proposition}\label{dicholemma2}
Let $\Gamma$ be an $h$-set satisfying the porosity condition, $0< p < q <\infty$,
$\bm{\tau}$ admissible. 
Assume that
\begin{equation}
\limsup \bm{\tau}^{-1} \bm{h}^{1/p} \paren{n}^{1/p} >0.
\label{dense-101}
\end{equation}
Then  $\mathcal{D}(\rn\setminus \Gamma)$ is dense in
$B^{\bm{\tau}}_{p,q}(\rn)$, therefore 
\begin{equation}
\tr{\Gamma} B^{\bm{\tau}}_{p,q}(\rn) \quad \mbox{cannot exist.}
\label{dense-91}
\end{equation}
\end{proposition}

\begin{proof}
Denote $\bm{\sigma} := \bm{\tau} \bm{h}^{-1/p} \paren{n}^{-1/p}$.
From \eqref{dense-101} it follows the existence of a constant $c>0$ and a strictly increasing sequence $(j_\ell)_{\ell \in \N} \subset \N$ such that 
\begin{equation}
\sigma_{j_\ell}^{-1} \geq c, \quad \ell \in \N.
\label{boundedaway}
\end{equation}

Given $k \in \N$, let $k_2 \in \N$ be such that
\begin{equation}
\frac{1}{k}+\frac{1}{k+1}\ldots+\frac{1}{k_2} \geq 2,
\label{harmon}
\end{equation}
which clearly exists. Then choose $j(k)$ large enough in order that it is possible to consider an optimal cover of $\Gamma_{2^{-j(k)}}$ in the sense of Remark \ref{optimalcover} with a number $N_{j(k)}$ of cubes such that $k_2^{-1} N_{j(k)} \geq 2$. Again, this is clearly possible, because of $N_{j(k)} \sim h_{j(k)}^{-1}$ (cf. Lemma \ref{optimcover}) and properties of gauge functions (cf. Definition \ref{defi-hset}). Actually, we shall also require that $j(k)$ be chosen in such a way it coincides with one of the $j_\ell$'s above, which is also possible.

Hence, for $i = k, k+1, \ldots, k_2$, $\whole{i^{-1}N_{j(k)}} > i^{-1}N_{j(k)} -1 \geq \frac{1}{2} i^{-1}N_{j(k)}$ and therefore, using \eqref{harmon},
$$
\whole{k^{-1}N_{j(k)}} + \whole{(k+1)^{-1}N_{j(k)}} + \ldots + \whole{k_2^{-1}N_{j(k)}} \geq N_{j(k)}.
$$

Let $k_1 \in \N$ be the smallest number such that
\begin{equation}
\whole{k^{-1}N_{j(k)}} + \whole{(k+1)^{-1}N_{j(k)}} + \ldots + \whole{k_1^{-1}N_{j(k)}} \geq N_{j(k)}.
\label{k1}
\end{equation}

On the other hand, following Lemma \ref{optimcover} and Discussion \ref{prep-dense},
\begin{equation}
\sum_{Q_{j(k)m} \in {\cal Q}(j(k))} \varphi_{j(k)m} \varphi = \varphi \quad \mbox{on } \, \Gamma_{2^{-j(k)}}.
\label{globalsum}
\end{equation}

Now we associate the terms of this sum in the following way: first consider the $m$'s such that $Q_{j(k)m}$ are grid subcubes of the first $\whole{k^{-1}N_{j(k)}}$ cubes of the optimal cover above; next consider the $m$'s such that $Q_{j(k)m}$ are grid subcubes of the following $\whole{(k+1)^{-1}N_{j(k)}}$ cubes of the same cover. And so on until the consideration of the $m$'s such that $Q_{j(k)m}$ are grid subcubes of the at most $\whole{k_1^{-1}N_{j(k)}}$ cubes of the optimal cover. This process certainly leads to repetition of grid cubes, so, in order not to affect the sum above we make the convention that $\varphi_{j(k)m}$ is replaced by zero any time it corresponds to a grid cube that has already been considered before. That is, $\varphi_{j(k)m}$ is replaced by $\tilde{\varphi}_{j(k)m}$ which can either be $\varphi_{j(k)m}$ or zero, according to what was just explained. 

Denote by $\psi_i$, $i=k, k+1, \ldots, k_1$, the sum of the $\tilde{\varphi}_{j(k)m}$'s corresponding to each association made above, so that the sum in \eqref{globalsum} can be written as
\begin{equation}
\sum_{i=k}^{k_1} \psi_i \varphi,
\label{assocsum}
\end{equation}

which, of course, still equals $\varphi$ on $\Gamma_{2^{-j(k)}}$.

The next thing to do is to choose, for each $i=k+1, \ldots, k_1$, $j(i)$ in such a way that
\begin{equation}
j(k) < j(k+1) < \ldots < j(k_1)
\label{orderedjays}
\end{equation}
and, moreover, each $j(i)$ coincides with one of the $j_\ell$'s used to get \eqref{boundedaway}. Then consider optimal covers of $\Gamma_{2^{-j(i)}}$ and the sum
\begin{equation}
\sum_{i=k}^{k_1} \sum_{Q_{j(i)m} \in {\cal Q}(j(i))} \varphi_{j(i)m} \psi_i \varphi,
\label{splitsum}
\end{equation}

which, by Lemma \ref{optimcover} and \eqref{orderedjays}, equals $\varphi$ on $\Gamma_{2^{-j(k_1)}}$.
We know, following Remark \ref{gridcover}, that each inner sum above has $\, \sim h_{j(i)}^{-1}$ terms. However, because of the product structure of the latter and the support of each $\psi_i$, actually only a smaller number of terms are non-zero. We claim that the latter number is $\, \lesssim i^{-1}  h_{j(i)}^{-1}$. This can be seen in the following way: for fixed $i=k,k+1,\ldots,k_1$, only the $\varphi_{j(i)m}$ such that 
$$
{\rm supp}\, \varphi_{j(i)m} \cap {\rm supp}\, \psi_i \not= \emptyset
$$
are of interest; notice that this implies that the cubes of the optimal cover of $\Gamma_{2^{-j(i)}}$ which contain the grid cubes at which such functions $\varphi_{j(i)m}$ are located are contained in a neighbourhood of radius $\, \sim 2^{-j(i)}$ of ${\rm supp}\, \psi_i$ which, in turn, is contained in the union of neighbourhoods of radius $\, \sim 2^{-j(k)}$ of the $\, \lesssim i^{-1} h_{j(k)}^{-1}$ cubes of the optimal cover of $\Gamma_{2^{-j(k)}}$ which were used to form $\psi_i$; since these union measures $\, \lesssim i^{-1} h_{j(k)}^{-1} h_{j(k)}$, that is, $\, \lesssim i^{-1}$, and the cubes of the optimal cover of $\Gamma_{2^{-j(i)}}$ contain disjoint balls of radius $\, \sim 2^{-j(i)}$, so measuring $\, \sim h_{j(i)}$ each, the number of such cubes that are being considered here must be $\, \lesssim i^{-1} h_{j(i)}^{-1}$; since, clearly, the same estimate holds for the family of grid subcubes of such cubes, our claim is proved.

We write then \eqref{splitsum} in the form
\begin{equation}
\sum_{i=k}^{k_1} \sum_{Q_{j(i)m} \in {\cal Q}(j(i))}' \varphi_{j(i)m} \psi_i \varphi,
\label{redsplitsum}
\end{equation}
where the prime over the sum means that we are in fact taking only $\, \lesssim i^{-1} h_{j(i)}^{-1}$ of the terms, according to the discussion above, without changing the value of \eqref{splitsum}.

Now we remark that, for $i=k,k+1,\ldots,k_1$,
$$
{\rm supp}\, \varphi_{j(i)m} \psi_i \subset {\rm supp}\, \varphi_{j(i)m} \subset \frac{3}{2} Q_{j(i)m}
$$
and, due to Leibniz formula, \eqref{orderedjays} and \eqref{derivcontrol}, for a fixed $K \in \N$ there is a positive constant $c_K$ such that
$$
|\Dd^\gamma(\varphi_{j(i)m} \psi_i)(x)| \leq c_K\, 2^{j(i)|\gamma|}, \quad x \in \Rn,\;\; \gamma \in \Non \, \mbox{ with } \,|\gamma| \leq K.
$$
That is, \eqref{coincidence}, \eqref{support} and \eqref{derivatives} hold for $\lambda_r = \lambda_{(j(i),m)} = 1$ and  $\varphi_r = \varphi_{(j(i),m)} = \varphi_{j(i)m} \psi_i$ with $r=(j(i),m)$ belonging to a set $I_k$ as in \eqref{Ik}, though more involved to describe: here we have $J_k = \{ j(k), j(k+1), \ldots,j(k_1) \}$ and each $M_{j(i)}$, with $i=k,k+1,\ldots,k_1$, is the set of $m$'s considered in the corresponding inner sum in \eqref{redsplitsum}.

It then follows, specially from \eqref{atomicuppest}, that the quasi-norm of the sum in \eqref{tends0} (or of an alternative similar sum, as discussed in Discussion \ref{prep-dense-2} apropos of the consideration of moment conditions) is
\begin{eqnarray*}
 & \lesssim & \left( \sum_{i=k}^{k_1} \left( \sum_{Q_{j(i)m} \in {\cal Q}(j(i))}' \tau_{j(i)}^p 2^{-j(i)n} \right)^{q/p} \right)^{1/q} \\
 & \lesssim & \left( \sum_{i=k}^{k_1} \tau_{j(i)}^q 2^{-j(i)n\frac{q}{p}} i^{-\frac{q}{p}} h_{j(i)}^{-\frac{q}{p}} \right)^{1/q} \\
 & = & \left( \sum_{i=k}^{k_1} i^{-\frac{q}{p}} \sigma_{j(i)}^q \right)^{1/q} \\
 & \lesssim & \left( \sum_{i=k}^\infty i^{-\frac{q}{p}} \right)^{1/q},
\end{eqnarray*}
the last estimate being a consequence of \eqref{boundedaway} and the choice of the $j(i)$ in \eqref{orderedjays}. 

Using now the hypothesis $0 < p < q < \infty$, we have that the last expression above tends to zero when $k$ tends to infinity, so that the result follows from Discussion \ref{prep-dense}.
\end{proof}

\begin{remark}
An observation corresponding to the one made in Remark \ref{noporosity} holds also here.
\end{remark}

The next result shows that when $1<q$ we can also conclude as in Proposition \ref{dicholemma} with an hypothesis weaker than \eqref{dense-10}. 

\begin{proposition}
Let $\Gamma$ be an $h$-set satisfying the porosity condition, $1<q<\infty$, $\bm{\tau}$ admissible. Assume that
\begin{equation}
\bm{\tau}^{-1} \bm{h}^{1/p} \paren{n}^{1/p} \notin \ell_{q'}.
\label{dense-2}
\end{equation}
Then
$\mathcal{D}(\rn\setminus \Gamma)$ is dense in $B^{\bm{\tau}}_{p,q}(\rn)$, therefore
$$\tr{\Gamma} B^{\bm{\tau}}_{p,q}(\rn) \quad \mbox{cannot exist}.$$
\label{theo-density}
\end{proposition}

\begin{proof}
Denote $\bm{\sigma} := \bm{\tau} \bm{h}^{-1/p} \paren{n}^{-1/p}$. From \eqref{dense-2} it follows the existence of some strictly increasing
sequence $(j_k)_{k\in\N} \subset \N$ such that
\begin{equation}
\sum_{l=j_k}^{j_{k+1}-1} \ \sigma_{l}^{-q'}  \geq \ 1\ , \quad k\in\mathbb{N}.
\label{dense-7}
\end{equation}
For each $k\in \N$, consider optimal covers of $\Gamma_{2^{-i}}$, in the sense of Remark \ref{optimalcover}, for all $i=j_\nu,j_{\nu}+1,\ldots,j_{\nu+1}-1$ and all $\nu=1,\ldots,k$, and  
follow our Discussions \ref{prep-dense} and \ref{prep-dense-2} with $I_k$ from \eqref{Ik} defined such that $J_k = \{ j\in \N : \, \exists \, \nu = 1,\ldots,k : j \in \{ j_\nu,j_{\nu}+1,\ldots,j_{\nu+1}-1 \} \}$ and $M_j = \{ m \in \Zn : \,  Q_{jm} \in {\cal Q}(j) \}$,
$\varphi_{(j,m)} = \varphi_{jm}$ and 
\begin{equation}
\lambda_{(j,m)} = 
\ds\frac{\sigma_j^{-q'} }{k} \left(\sum_{l=j_\nu}^{j_{\nu+1}-1}
\sigma_l^{-q'}\right)^{-1}, \quad j= j_\nu , \dots, j_{\nu+1}-1, \ \nu = 1, \dots,
k. 
\label{dense-8}
\end{equation}
Due to Remark \ref{partitions}, this fits nicely into Discussions \ref{prep-dense} and \ref{prep-dense-2}. In particular, \eqref{support} and \eqref{derivatives} immediately hold. As to \eqref{coincidence}, we have, for $x\in \Gamma_{2^{-j_{k+1}}}$, that
\begin{align*}
 \lefteqn{\sum_{j=j_1}^{j_{k+1}-1} \sum_{Q_{jm} \in {\cal Q}(j)}\lambda_{(j,m)} \varphi_{(j,m)}(x) \varphi(x)} \\
= \ & \frac1k \ \sum_{\nu=1}^k \sum_{j=j_\nu}^{j_{\nu+1}-1}  \sigma_j^{-q'} \left(\sum_{l=j_
\nu}^{j_{\nu+1}-1}
\sigma_l^{-q'}\right)^{-1} \sum_{Q_{jm} \in {\cal Q}(j)} \varphi_{jm}(x) \varphi(x) \\
= \ & \frac{\varphi(x)}{k} \ \sum_{\nu=1}^k \sum_{j=j_\nu}^{j_{\nu+1}-1}  \sigma_j^{-q'} \left(\sum_{l=j_
\nu}^{j_{\nu+1}-1}
\sigma_l^{-q'}\right)^{-1}  \ = \ \varphi(x).
\end{align*}

It then follows, specially from \eqref{atomicuppest}, that the quasi-norm of the sum in \eqref{tends0} (or of an alternative similar sum, as discussed in Discussion \ref{prep-dense-2} apropos of the consideration of moment conditions) is
\begin{align*}
 \lesssim \ & \left(\sum_{j=j_1}^{j_{k+1}-1}  \left( \sum_{Q_{jm} \in {\cal Q}(j)}
|\lambda_{(j,m)} \tau_j 2^{-j\frac{n}{p}}|^p\right)^{\frac{q}{p}}\right)^{\frac1q} \\ 
 = \ & \left(\sum_{j=j_1}^{j_{k+1}-1} \lambda_{(j,m)}^q   \tau_j^q 2^{-j\frac{n}{p}q} (\sharp {\cal Q}(j))^{\frac{q}{p}} \right)^{\frac1q} \\
 \sim \ & \left(\sum_{\nu=1}^{k}\ \sum_{j=j_\nu }^{j_{\nu+1}-1}
\frac{\sigma_j^{-q'q}}{k} \left(\sum_{l=j_\nu}^{j_{\nu+1}-1}
\sigma_l^{-q'}\right)^{-q} \sigma_j^q   \right)^{\frac1q} \\
 = \ & \ \frac{1}{k} \left(\sum_{\nu=1}^{k} \left(\sum_{l=j_\nu}^{j_{\nu+1}-1}
\sigma_l^{-q'}\right)^{1-q}\right)^{\frac1q} \\
\leq \ & \ \frac{1}{k} \left( \sum_{\nu=1}^{k}
1\right)^{\frac1q}\quad = \  k^{-\frac {1}{q'}}
\ \xrightarrow[k\to\infty]{} \ 0, 
\end{align*}
since $q>1$, where we have also used Remark \ref{gridcover}, \eqref{dense-8} and \eqref{dense-7}. Thus \eqref{tends0} holds (possibly with an alternative similar sum if moment conditions are required, as mentioned above) and Discussion \ref{prep-dense} then concludes the proof.
\end{proof}

\begin{remark}
An observation corresponding to the one made in Remark \ref{noporosity} holds also here.
\end{remark}

\begin{theorem}
\label{coroleither}
Let $\Gamma$ be an $h$-set satisfying the porosity condition, $\bm{\sigma}$ an admissible sequence and let either $1\leq p<\infty$,
$0<q<\infty$, or $0<q\leq p<1$. Then {\bfseries either}
\begin{align*}
\text{\upshape\bfseries (i)} &\quad \Bsg = \tr{\Gamma} \Bsih\quad \mbox{exists}
\intertext{\bfseries or}
\text{\upshape\bfseries (ii)} &\quad  \mathcal{D}(\rn\setminus \Gamma)\quad \mbox{is dense in}\quad
\Bsih\\
& \quad \mbox{and, therefore,}\quad \tr{\Gamma} \Bsih\ \mbox{cannot exist}.
\end{align*}
\end{theorem}

\begin{proof}
Either \eqref{ext-2-4-a} holds or not. If it holds, then from Definition \ref{defi-Bsg} it follows that (i) above holds. If \eqref{ext-2-4-a} fails, then, with $\bm{\tau} := \bm{\sigma} \bm{h}^{1/p} \paren{n}^{1/p}$,
$$\bm{\tau}^{-1} \bm{h}^{1/p} \paren{n}^{1/p} \notin \ell_{q'},$$
therefore, from Propositions \ref{dicholemma} (in the case $0<q \leq 1$) and \ref{theo-density} (in the case $1<q<\infty$), (ii) above holds.
\end{proof}

\begin{remark}
The only use of the porosity condition in the proof above is in guaranteeing the existence of suitable atoms when moment conditions for these are required. Therefore the result of the theorem above holds without assuming porosity of $\Gamma$ whenever the atomic representation of $\Bsih$ does not require atoms with moment conditions.
\end{remark}

\begin{remark}
\label{remeither}
From the proof above it also follows that, under the conditions of the theorem, (i) is equivalent to \eqref{ext-2-4-a}, that is, $\bm{\sigma}^{-1} \in \ell_{q'}.$
\end{remark}

\begin{conjecture}\label{conj}
Let $\Gamma$ be an $h$-set satisfying the porosity condition, $\bm{\sigma}$ an admissible sequence and $0<p<1$, $0<p<q<\infty$. Define $v_p$ by the identity $\frac{1}{v_p} = \frac{1}{p}-\frac{1}{q}$. With the extra assumption 
\begin{equation}
\lim_{j \to \infty} h_j \sigma_j^{v_p} = 0
\label{extracond}
\end{equation}
the alternative in the conclusion of the theorem above also holds. More precisely, assertion {\upshape\bfseries (i)} holds iff assertion {\upshape\bfseries (ii)} does not hold iff 
\begin{equation}
\bm{\sigma}^{-1} \in \ell_{v_p}.
\label{condconj}
\end{equation}
\end{conjecture}

We discuss a little bit this conjecture. 
Clearly, if \eqref{condconj} holds then Definition \ref{defi-Bsg} guarantees that (i) holds, and therefore (ii) does not, even without the extra assumption \eqref{extracond}.

On the other hand, Proposition \ref{dicholemma2}, with $\bm{\tau} := \bm{\sigma} \bm{h}^{1/p} \paren{n}^{1/p}$, guarantees that at least in the special subcase of $\bm{\sigma}^{-1} \notin \ell_{v_p}$ given by $\limsup \bm{\sigma}^{-1} > 0$ (ii) holds, and therefore (i) does not.

In other words, we can get the alternative in the conclusion of the theorem above in the case  $0<p<1$, $0<p<q<\infty$ if instead of the extra assumption \eqref{extracond} we assume that $\limsup \bm{\sigma}^{-1} > 0$. The drawback of this restriction is that it immediately implies that \eqref{condconj} never holds, so \eqref{condconj} is not a real alternative under that extra assumption. From this point of view, \eqref{extracond} is more interesting. Notice also that, in order that this conjecture does not contradict Definition \ref{defi-Bsg}, under \eqref{extracond} it must be true that \eqref{condconj} holds whenever any of the conditions \eqref{ext-2-4-b} does. In fact, by standard comparison criteria for series it is easy to see directly that this is indeed the case (take also into account that the case $v_r=\infty$ in the conditions \eqref{ext-2-4-b} never holds in presence of \eqref{extracond}).

\begin{remark}
It is worth noticing that in the case when $\Gamma$ is a $d$-set with $0<d<n$, \eqref{extracond} holds whenever \eqref{condconj} fails, so in that setting, and by the above discussion, there is no need to impose the extra condition \eqref{extracond}, the conjecture then turning out to be indeed a known result (cf. \cite[(1.5)]{T-dicho}). However, for general $h$-sets we cannot dispense with an extra assumption (such as (\ref{extracond})), as the Conjecture above then fails, at least in the part of the equivalence with (\ref{condconj}). We give a class of examples where this is the case:

Consider $h(r)=(1+|\log r|)^b, \, r\in (0,1]$, for any given $b \in (-1,0)$, recall \eqref{h-log-ex} and the discussion afterwards. Let $0<p<1$ and $0<p<q<\infty$. Given any $\kappa \in (-b(\frac{1}{p} - \frac{1}{q}) + (1+b) \frac{1}{q'}, \frac{1}{p} - \frac{1}{q}]$, consider $\bm{\sigma}=((1+j)^\kappa)_j$. It is easy to see that, though (\ref{condconj}) fails, by Definition \ref{defi-Bsg}(ii) with $r=\min\{q,1\}$ (i) of the theorem above holds, i.e., the trace exists (and therefore the corresponding assertion (ii) fails, as was discussed in the beginning of Section \ref{sect-2-2}).

Nevertheless, our conjecture as stated above resists this class of counterexamples: as is easily seen, for such a class the assumption (\ref{extracond}) is violated.
\end{remark}

\subsection{Dichotomy results}
We combine our results from the previous subsections and deal with the so-called {\em dichotomy} of trace spaces. First we briefly describe the idea.

Recall that $\DRn$ is dense in all spaces $\Btau(\rn)$ with $0<p,q<\infty$. So removing from $\rn$ only `small enough' $\Gamma$ one can ask whether (still) 
\begin{equation}
\text{$\DG$ is dense in $\Btau(\rn)$.}
\label{dgamma-dense}
\end{equation} 
Conversely, we have the affirmative trace results mentioned in Section~\ref{sect-2-1}, but one can also ask for what (`thick enough') $\Gamma$ 
\begin{equation}
\text{there exists a trace of $\Btau(\rn)$ in  $L_p(\Gamma)$}
\label{trace-exist}
\end{equation}
(for sufficiently high smoothness and $q$-regularity). Though these questions may arise independently, it is at least clear that whenever $\DG$ is dense in $\Btau(\rn)$, then there cannot exists a trace according to \eqref{trace-def}; see our discussion at the beginning of Section~\ref{sect-2-2} and \cite{T-dicho} for the corresponding argument in the classical case.

\begin{remark}
It is not always true that one really has an alternative in the sense that {\em either} there is a trace {\em or} $\DG$ is dense in $\Btau(\rn)$. Triebel studied such questions in \cite{T-dicho} for spaces of type $\B$ and described an example of a set $\Gamma$ where a gap remains: traces can only exist for spaces $\B$ with smoothness $s\geq s_0$, whereas density requires $s\leq s_1$ and $s_1< s_0$. 
\end{remark}

However, if one obtains an alternative between \eqref{dgamma-dense} and \eqref{trace-exist}, then following Triebel in \cite{T-dicho} we call this phenomenon {\em dichotomy}. First we recall this notion for spaces of type $\B(\rn)$ and point out necessary modifications for our setting afterwards. Let 
\begin{equation}
\tr{\Gamma} : \B(\rn) \rightarrow L_p(\Gamma)
\end{equation}
be the trace operator defined by completion from the pointwise trace according to \eqref{trace-def}, and 
\begin{equation}
B_p(\rn) = \{\B(\rn): 0<q<\infty,  s\in\real\},\quad 0<p<\infty.
\label{dicho-1}
\end{equation}
\begin{definition}\label{defi-dicho}
Let $n\in \nat$, $\Gamma\subset\rn$, $0<p<\infty$. 
The {\em dichotomy} of the scale $B_p(\rn)$ with respect to $L_p(\Gamma)$, denoted by $\dicho(B_p(\rn),L_p(\Gamma))$, is defined by
\begin{align}
& \dicho(B_p(\rn), L_p(\Gamma)) = (\dc{s}, \dc{q}),\quad \dc{s}\in\real,\ 0<\dc{q}<\infty, 
\label{dicho-2}
\intertext{if}
\text{\upshape\bfseries (i)}& \ \tr{\Gamma} B^s_{p,q}(\Rn) \ \text{exists for}\ \begin{cases}
s>\dc{s}, & 0<q<\infty,\\ s= \dc{s},& 0<q\leq \dc{q},
\end{cases}\label{dicho-3a} 
\intertext{and}
\text{\upshape\bfseries (ii)}& \ \mathcal{D}(\rn\setminus \Gamma) \ \text{is dense
  in $\B(\rn)$ for}\ \begin{cases}
s=\dc{s}, & \dc{q} <q<\infty,\\ s< \dc{s},& 0<q<\infty.
\end{cases}
\label{dicho-3b}
\end{align}
\end{definition}

\begin{remark} \label{R7.67}
The notion applies to spaces of type $\F$ in a similar way, cf. \cite{T-dicho}. Then one has to define the  borderline cases $\dc{q} =0$ and $\dc{q} =\infty$, too. But this will not be needed at the moment in our setting. \\
We briefly explain why it is reasonable to look for the `breaking point' $(\dc{s}, \dc{q})$. In the diagram below we sketch this situation, where spaces of type $\B(\rn)$ are indicated by their parameters $(\frac1q,s)$ (while $p$ is always assumed to be fixed). Assume that $\DG$ is dense in some space $B^{s_2}_{p,q_2}(\rn)$.\\ 
\begin{minipage}{0.48\textwidth}

  Then, since $\DRn$ is dense in all spaces $B^t_{p,u}(\rn)$, we immediately obtain that $\DG$ is also dense in all spaces $B^t_{p,u}(\rn)$ in which $B^{s_2}_{p,q_2}(\rn)$ is continuously embedded (the shaded area left and below of $(\frac{1}{q_2},s_2)$ referring to $B^{s_2}_{p,q_2}(\rn)$ in the diagram). This explains why we look for the largest possible $s_2$ and smallest possible $q_2$ in \eqref{dicho-3b}. Conversely, if the trace exists for some space  $B^{s_1}_{p,q_1}(\rn)$, then it exists likewise for \end{minipage}~~~~\begin{minipage}{0.47\textwidth}
\begin{picture}(0,0)%
\includegraphics{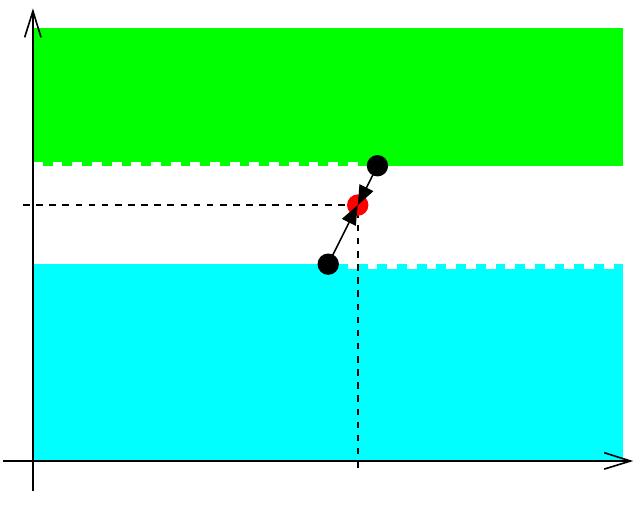}%
\end{picture}%
\setlength{\unitlength}{4144sp}%
\begingroup\makeatletter\ifx\SetFigFont\undefined%
\gdef\SetFigFont#1#2#3#4#5{%
  \reset@font\fontsize{#1}{#2pt}%
  \fontfamily{#3}\fontseries{#4}\fontshape{#5}%
  \selectfont}%
\fi\endgroup%
\begin{picture}(2907,2392)(751,-2150)
\put(1441,-151){\makebox(0,0)[lb]{\smash{{\SetFigFont{10}{12.0}{\familydefault}{\mddefault}{\updefault}{\color[rgb]{0,0,0}$\tr{\Gamma}$ exists}%
}}}}
\put(1036,-1681){\makebox(0,0)[lb]{\smash{{\SetFigFont{10}{12.0}{\familydefault}{\mddefault}{\updefault}{\color[rgb]{0,0,0}$\DG$ dense}%
}}}}
\put(2476,-421){\makebox(0,0)[lb]{\smash{{\SetFigFont{10}{12.0}{\familydefault}{\mddefault}{\updefault}{\color[rgb]{0,0,0}$B^{s_1}_{p,q_1}(\rn)$}%
}}}}
\put(2251,-1141){\makebox(0,0)[rb]{\smash{{\SetFigFont{10}{12.0}{\familydefault}{\mddefault}{\updefault}{\color[rgb]{0,0,0}$B^{s_2}_{p,q_2}(\rn)$}%
}}}}
\put(766,119){\makebox(0,0)[rb]{\smash{{\SetFigFont{10}{12.0}{\familydefault}{\mddefault}{\updefault}{\color[rgb]{0,0,0}$s$}%
}}}}
\put(3601,-2086){\makebox(0,0)[rb]{\smash{{\SetFigFont{10}{12.0}{\familydefault}{\mddefault}{\updefault}{\color[rgb]{0,0,0}$\frac1q$}%
}}}}
\put(2386,-2086){\makebox(0,0)[b]{\smash{{\SetFigFont{10}{12.0}{\familydefault}{\mddefault}{\updefault}{\color[rgb]{0,0,0}$\frac{1}{\dc{q}}$}%
}}}}
\put(811,-691){\makebox(0,0)[rb]{\smash{{\SetFigFont{10}{12.0}{\familydefault}{\mddefault}{\updefault}{\color[rgb]{0,0,0}$\dc{s}$}%
}}}}
\end{picture}%

\end{minipage}

\noindent 
 all spaces which embed continuously into $B^{s_1}_{p,q_1}(\rn)$ (the shaded area right and above of $(\frac{1}{q_1},s_1)$ referring to $B^{s_1}_{p,q_1}(\rn)$ in the diagram); hence we now search for the smallest possible $s_1$ and largest possible $q_1$ in \eqref{dicho-3a}.
Dichotomy in the above-defined sense happens, if the two `extremal' points merge, that is, the common breaking point $(\frac{1}{\dc{q}}, \dc{s})$ in the diagram exists. Then we denote the couple of parameters $(\dc{s}, \dc{q})$ by $\dicho(B_p(\rn),L_p(\Gamma))$.\\
In \cite[Sect.~6.4.3]{T-wave-dom} Triebel mentioned already, that it might be more reasonable in general to exclude the limiting case $q=\dc{q}$ in \eqref{dicho-3a} or shift it to \eqref{dicho-3b}. But as will turn out below, (also) in our context the breaking point $q=\dc{q}$ is always on the trace side.
\end{remark}

Now we collect what is known in the situation of spaces $\B$ for hyperplanes $\Gamma=\real^m$ or $d$-sets $\Gamma$, $0<d<n$.

\begin{proposition}\label{prop-dicho-Triebel}
Let $0<p<\infty$.\\[-4ex]
\bli
\item[{\upshape\bfseries (i)}]
Let $m\in\nat$, $m\leq n-1$. Then
\begin{align*}
\dicho(B_p(\rn), L_p(\real^m)) & = \left(\frac{n-m}{p},
\min\{p,1\}\right).
\end{align*}
\item[{\upshape\bfseries (ii)}]
Let $\Gamma$ be a compact $d$-set, $0<d<n$. Then
\begin{align*}
\dicho(B_p(\rn), L_p(\Gamma)) & = \left(\frac{n-d}{p},
\min\{p,1\}\right).
\end{align*}
\eli
\end{proposition}

\begin{remark}
The result (i) is proved in this explicit form in \cite{T-dicho}, see also \cite[Cor.~6.69]{T-wave-dom} and \cite{CS-4}. The second part (ii) can be found in \cite{T-dicho} and \cite[Thm.~6.68]{T-wave-dom} (with forerunners in \cite[Thm.~17.6]{T-Frac} and \cite[Prop.~19.5]{T-func}).
As for the classical case of a bounded $C^\infty$ domain $\Omega$ in $\rn$ 
with boundary $\Gamma=\partial\Omega$ (referring to $d=n-1$), $1<p<\infty$, $1\leq
q<\infty$, $s\in\real$, the situation (whether $\mathcal{D}(\Omega)$ is dense
in $B^s_{p,q}(\Omega)$) is known for long, cf. \cite[Thm.~4.7.1]{T-I}, even in
a more general setting, we refer to \cite{CS-trace-dom}. Moreover, we dealt with dichotomy questions for weighted spaces corresponding to (ii) in \cite{iwona-2,Ha-EE}. More precisely, when $\Gamma$ is again a compact $d$-set, $0<d<n$, $0<p<\infty$, and the weight $\wGamma$ is given by 
\[
\wGamma(x)= \begin{cases} \dist (x,\Gamma)^{\varkappa} ,&
 \text{if}\quad  \dist(x,\Gamma) \leq 1,   \\ 1 , & \text{if}\quad \dist(x,\Gamma) \geq 1, \end{cases}
\]
with $\varkappa>-(n-d)$, then
\begin{equation}\label{Ha-EE-d}
\dicho(B_p(\rn, \wGamma), L_p(\Gamma))  = \left(\frac{n-d+\varkappa}{p},
\min\{p,1\}\right).
\end{equation}
In all the cases mentioned above there are parallel results for $F$-spaces, too.
\end{remark}

We return to our setting to study traces of $\Btau(\rn)$ on fractal $h$-sets. Obviously, definitions \eqref{dicho-1}--\eqref{dicho-3b} have to be adapted now. The
following extension seems appropriate. (Recall that we put $v=\infty$
if $\frac1v=0$.)  

\begin{definition}\label{defi-dicho-ex}
Let $\Gamma$ be a porous $h$-set, and 
\begin{equation}
B_p(\rn) = \{\Btau(\rn) : 0<q<\infty, \ \bm{\tau}\ \text{admissible}\},\quad 0<p<\infty.
\label{dicho-6}
\end{equation}
Then the {\em dichotomy} of the scale
$B_p(\rn)$ with respect to $L_p(\Gamma)$ is defined by
\begin{equation}
\mathbb{D}(B_p(\rn), L_p(\Gamma)) = (\dc{\bm{\tau}}, \dc{q}),
\label{dicho-7}
\end{equation}
if $\dc{\bm{\tau}}$ is admissible, $0<\dc{q} < \infty$, and 
\begin{align}
\text{\upshape\bfseries (i)}& \quad \tr{\Gamma} \Btau(\Rn) \ \text{exists for}\ 
\bm{\tau}^{-1} \dc{\bm{\tau}} \in \ell_v, \ \text{where}\ \frac1v =
\left(\frac{1}{\dc{q}}-\frac1q\right)_+ ,\label{dicho-8a}
\intertext{and}
\text{\upshape\bfseries (ii)}& \quad \mathcal{D}(\rn\setminus \Gamma) \ \text{is dense
  in $\Btau(\rn)$ for}\ \bm{\tau}^{-1} \dc{\bm{\tau}} \not\in \ell_v .
\label{dicho-8b}
\end{align}
\end{definition}

\begin{remark}\label{dicho-new-d}
One immediately verifies that \eqref{dicho-6}--\eqref{dicho-8b} with
$h(r)=r^d$, $\bm{\tau} = \paren{s}$ for $s\in\real$, 
and $0<p<\infty$, $0<q<\infty$, coincides with
\eqref{dicho-1}--\eqref{dicho-3b}, that is, the notion is extended (and we may thus and in this sense retain the same symbols in a slight abuse of notation). Using  the continuous embedding
\begin{equation}\label{BB-Rn}
B^{\bm{\sigma}}_{p,q_1}(\rn) \hookrightarrow
B^{\bm{\tau}}_{p,q_2}(\rn)
\end{equation}
for $\bm{\sigma}^{-1}\bm{\tau}\in \ell_{q^\ast}$ with $\frac{1}{q^\ast} = (\frac{1}{q_2} - \frac{1}{q_1})_+$, cf. \cite[Thm.~3.7]{C-Fa}, we can argue as in Remark~\ref{R7.67} to motivate the definition.
\end{remark}

Part of the results contained in Theorem~\ref{coroleither} and Remark \ref{remeither} can then be rephrased in terms of this notion of dichotomy in the following way.

\begin{corollary}\label{Coro-2.28}
Let $\Gamma$ be an $h$-set satisfying the porosity condition and $1\leq p<\infty$. Then
$$
\mathbb{D}(B_p(\rn), L_p(\Gamma)) = \left( \bm{h}^{1/p} \paren{n}^{1/p}, \ 1
\right).
$$
\end{corollary}

\begin{remark}
Again, with the special setting $h(r)=r^d$, $\bm{\tau} = \paren{s}$, Corollary~\ref{Coro-2.28} coincides with Proposition~\ref{prop-dicho-Triebel}(ii) for $p\geq 1$. So one might expect some parallel result with $\dc{q}=p$ for $0<p<1$, see also \eqref{Ha-EE-d}. However, we are not yet able to (dis)prove this claim. More precisely, if $0<p<1$ and $p<q<\infty$, \eqref{dicho-8a} is satisfied with $\dc{\bm{\tau}}=\bm{h}^{1/p} \paren{n}^{1/p}$ and $\dc{q}=p$, see \eqref{ext-2-4-b}. The gap that remains at the moment is to confirm \eqref{dicho-8b} in that case (possibly with some additional assumptions), recall Conjecture~\ref{conj} and the discussion afterwards.  
\end{remark}

But we can obtain some weaker version as follows. Introducing a notion of dichotomy where also $q$ can be fixed beforehand, hence adapting accordingly  \eqref{dicho-6}, \eqref{dicho-7} and denoting the new versions respectively by $B_{p,q}(\Rn)$ and $\mathbb{D}(B_{p,q}(\rn), L_p(\Gamma)) = (\dc{\bm{\tau}}, \dc{q})$, while \eqref{dicho-8a} and \eqref{dicho-8b} are kept unchanged, we can also cast the case $0<q\leq p<1$ of Theorem~\ref{coroleither} and Remark \ref{remeither} in terms of dichotomy.

\begin{corollary}
Let $\Gamma$ be an $h$-set satisfying the porosity condition and $0<q\leq p<1$. Then
$$
\mathbb{D}(B_{p,q}(\rn), L_p(\Gamma)) = \left( \bm{h}^{1/p} \paren{n}^{1/p}, \ p
\right).
$$
\end{corollary}

\subsection*{Acknowledgements}
It is our pleasure to thank our colleagues in the Department of
Mathematics at the University of Coimbra for their kind hospitality during our stays there.

The first author was partially supported by {\itshape
      FEDER} funds through {\itshape COM\-PETE}--Operational Programme
    Factors of Competitiveness (``Programa Operacional Factores de
    Competitividade'') and by Portuguese funds through the {\itshape Center
      for Research and Development in Mathematics and Applications}
    (University of Aveiro) and the Portuguese Foundation for Science
    and Technology (``FCT--Funda\c c\~ao para a Ci\^encia e a
    Tecnologia''), within project PEst-C/MAT/UI4106/2011 with COMPETE
    number FCOMP-01-0124-FEDER-022690 and projects PEst-OE/MAT/UI4106/2014 and \linebreak UID/MAT/04106/2013.
		The second author was also supported by the DFG Heisenberg fellowship HA 2794/1-2.




\begin{thebibliography}{BGT87}

\bibitem[BGT]{BGT}
N.H. Bingham, C.M. Goldie, and J.L. Teugels.
\newblock {\em Regular variation}, volume~27 of {\em Encyclopedia of
  Mathematics and its Applications}.
\newblock Cambridge Univ. Press, Cambridge, 1987.

\bibitem[Br1]{bricchi-4}
M.~Bricchi.
\newblock Existence and properties of $h$-sets.
\newblock {\em Georgian Math. J.}, 9(1):13--32, 2002.

\bibitem[Br2]{bricchi-diss}
M.~Bricchi.
\newblock {\em Tailored function spaces and related $h$-sets}.
\newblock PhD thesis, Friedrich-Schiller-Universit\"at Jena, Germany, 2002.

\bibitem[Br3]{bricchi-fsdona}
M.~Bricchi.
\newblock Complements and results on $h$-sets.
\newblock In D.D. Haroske, Th. Runst, and H.J. Schmei{\ss}er, editors, {\em
  Function Spaces, Differential Operators and Nonlinear Analysis - The Hans
  Triebel Anniversary Volume}, pages 219--230. Birkh{\"a}user, Basel, 2003.

\bibitem[Br4]{bricchi-3}
M.~Bricchi.
\newblock Tailored {B}esov spaces and $h$-sets.
\newblock {\em Math. Nachr.}, 263-264(1):36--52, 2004.

\bibitem[Ca1]{Cae-00}
A.~Caetano.
\newblock Approximation by functions of compact support in
  {B}esov-{T}riebel-{L}izorkin spaces on irregular domains.
\newblock {\em Studia Math.}, 142(1):47--63, 2000.

\bibitem[Ca2]{Cae-env-h}
A.M. Caetano.
\newblock Growth envelopes of {B}esov spaces on fractal $h$-sets.
\newblock {\em Math. Nachr.}, 286(5-6):550--568, 2013.

\bibitem[CaF]{C-Fa}
A.M. Caetano and W.~Farkas.
\newblock Local growth envelopes of {B}esov spaces of generalized smoothness.
\newblock {\em Z. Anal. Anwendungen}, 25:265--298, 2006.

\bibitem[CaL1]{CL-3}
A.M. Caetano and S.~Lopes.
\newblock The fractal {D}irichlet {L}aplacian.
\newblock {\em Rev. Mat. Complut.}, 24(1):189--209, 2011.

\bibitem[CaL2]{CL-2}
A.M. Caetano and S.~Lopes.
\newblock Spectral theory for the fractal {L}aplacian in the context of
  {$h$}-sets.
\newblock {\em Math. Nachr.}, 284(1):5--38, 2011.

\bibitem[CoF]{C-F}
F.~Cobos and D.L. Fernandez.
\newblock {H}ardy-{S}obolev spaces and {B}esov spaces with a function
  parameter.
\newblock In M.~Cwikel and J.~Peetre, editors, {\em Function spaces and
  applications}, volume 1302 of {\em Lecture Notes in Mathematics}, pages
  158--170. Proc. US-Swed. Seminar held in Lund, June, 1986, Springer, 1988.


\bibitem[EKP]{EKP}
D.E. Edmunds, R.~Kerman, and L.~Pick.
\newblock Optimal {S}obolev imbeddings involving rearrangement-invariant
  quasinorms.
\newblock {\em J. Funct. Anal.}, 170:307--355, 2000.

\bibitem[ET1]{ET}
D.E. Edmunds and H.~Triebel.
\newblock {\em Function spaces, entropy numbers, differential operators}.
\newblock Cambridge Univ. Press, Cambridge, 1996.

\bibitem[ET2]{ET6}
D.E. Edmunds and H.~Triebel.
\newblock Spectral theory for isotropic fractal drums.
\newblock {\em C. R. Acad. Sci. Paris}, 326(11):1269--1274, 1998.

\bibitem[ET3]{ET7}
D.E. Edmunds and H.~Triebel.
\newblock Eigenfrequencies of isotropic fractal drums.
\newblock {\em Oper. Theory Adv. Appl.}, 110:81--102, 1999.

\bibitem[Fa]{falconer}
K.J. Falconer.
\newblock {\em The geometry of fractal sets}.
\newblock Cambridge Univ. Press, Cambridge, 1985.

\bibitem[FaLe]{F-L}
W.~Farkas and H.-G. Leopold.
\newblock Characterisation of function spaces of generalised smoothness.
\newblock {\em Ann. Mat. Pura Appl.}, 185(1):1 -- 62, 2006.

\bibitem[Ha]{Ha-EE}
D.D. Haroske.
\newblock Dichotomy in {M}uckenhoupt weighted function space: {A} fractal
  example.
\newblock In B.M. Brown, J.~Lang, and I.~Wood, editors, {\em {Spectral Theory,
  Function Spaces and Inequalities. New Techniques and Recent Trends}}, volume
  219 of {\em Operator Theory: Advances and Applications}, pages 69--89.
  Springer, Basel, 2012.

\bibitem[JW]{JW-1}
A.~Jonsson and H.~Wallin.
\newblock Function spaces on subsets of {$\rn$}.
\newblock Math. Rep. Ser. 2, No.1, xiv + 221 p., 1984.

\bibitem[KL]{Kal-Liz}
G.A. Kalyabin and P.I. Lizorkin.
\newblock Spaces of functions of generalized smoothness.
\newblock {\em Math. Nachr.}, 133:7--32, 1987.

\bibitem[KZ]{K-Z}
V.~Knopova and M.~Z{\"a}hle.
\newblock Spaces of generalized smoothness on $h$-sets and related {D}irichlet
  forms.
\newblock {\em Studia Math.}, 174(3):277--308, 2006.

\bibitem[Li]{lizorkin}
P.I. Lizorkin.
\newblock Spaces of generalized smoothness.
\newblock {\em Mir, Moscow}, pages 381--415, 1986.
\newblock Appendix to Russian ed. of \cite{T-F1}; Russian.

\bibitem[Lo]{lopes-diss}
S.~Lopes.
\newblock {\em Besov spaces and the Laplacian on fractal $h$-sets}.
\newblock PhD thesis, Universidade de Aveiro, Portugal, 2009.

\bibitem[Ma]{mattila}
P.~Mattila.
\newblock {\em Geometry of sets and measures in euclidean spaces}.
\newblock Cambridge Univ. Press, Cambridge, 1995.

\bibitem[Me]{merucci}
C.~Merucci.
\newblock Applications of interpolation with a function parameter to {L}orentz,
  {S}obolev and {B}esov spaces.
\newblock In M.~Cwikel and J.~Peetre, editors, {\em Interpolation spaces and
  allied topics in analysis}, volume 1070 of {\em Lecture Notes in
  Mathematics}, pages 183--201. Proc. Conf., Lund/Swed. 1983, Springer, 1984.

\bibitem[Mo1]{moura}
S.D. Moura.
\newblock Function spaces of generalised smoothness.
\newblock {\em Dissertationes Math.}, 398:88 pp., 2001.

\bibitem[Mo2]{moura-diss}
S.D. Moura.
\newblock {\em Function spaces of generalised smoothness, entropy numbers,
  applications}.
\newblock PhD thesis, Universidade de Coimbra, Portugal, 2002.

\bibitem[Ne1]{neves-01}
J.S. Neves.
\newblock {L}orentz-{K}aramata spaces, {B}essel and {R}iesz potentials and
  embeddings.
\newblock {\em Dissertationes Math.}, 405:46 pp., 2002.

\bibitem[Ne2]{neves-02}
J.S. Neves.
\newblock Spaces of {B}essel-potential type and embeddings: the super-limiting
  case.
\newblock {\em Math. Nachr.}, 265:68--86, 2004.

\bibitem[Pi]{iwona-2}
I.~Piotrowska.
\newblock Traces on fractals of function spaces with {M}uckenhoupt weights.
\newblock {\em Funct. Approx. Comment. Math.}, 36:95--117, 2006.

\bibitem[Ro]{Rogers}
C.A. Rogers.
\newblock {\em Hausdorff measures}.
\newblock Cambridge Univ. Press, London, 1970.

\bibitem[Sch1]{CS-4}
C.~Schneider.
\newblock Trace operators in {B}esov and {T}riebel-{L}izorkin spaces.
\newblock {\em Z. Anal. Anwendungen}, 29(3):275--302, 2010.

\bibitem[Sch2]{CS-trace-dom}
C.~Schneider.
\newblock Traces of {B}esov and {T}riebel-{L}izorkin spaces on domains.
\newblock {\em Math. Nachr.}, 284(5-6):572--586, 2011.

\bibitem[Tr1]{T-I}
H.~Triebel.
\newblock {\em Interpolation theory, function spaces, differential operators}.
\newblock North-Holland, Amsterdam, 1978.

\bibitem[Tr2]{T-F1}
H.~Triebel.
\newblock {\em Theory of function spaces}.
\newblock Birkh\"auser, Basel, 1983.
\newblock Reprint (Modern Birkh\"auser Classics) 2010.

\bibitem[Tr3]{T-F2}
H.~Triebel.
\newblock {\em Theory of function spaces {II}}.
\newblock Birkh\"auser, Basel, 1992.
\newblock Reprint (Modern Birkh\"auser Classics) 2010.

\bibitem[Tr4]{T-Frac}
H.~Triebel.
\newblock {\em Fractals and spectra}.
\newblock Birkh\"auser, Basel, 1997.
\newblock Reprint (Modern Birkh\"auser Classics) 2011.

\bibitem[Tr5]{T-func}
H.~Triebel.
\newblock {\em The structure of functions}.
\newblock Birkh\"auser, Basel, 2001.

\bibitem[Tr6]{T-F3}
H.~Triebel.
\newblock {\em Theory of function spaces {III}}.
\newblock Birkh\"auser, Basel, 2006.

\bibitem[Tr7]{T-dicho}
H.~Triebel.
\newblock The dichotomy between traces on $d$-sets {$\Gamma$} in
  {$\mathbb{R}^n$} and the density of
  {$\mathcal{D}(\mathbb{R}^n\setminus\Gamma)$} in function spaces.
\newblock {\em Acta Math. Sinica}, 24(4):539--554, 2008.

\bibitem[Tr8]{T-wave-dom}
H.~Triebel.
\newblock {\em Function Spaces and Wavelets on domains}.
\newblock EMS Tracts in Mathematics (ETM). European Mathematical Society (EMS),
  Z\"urich, 2008.

\bibitem[TW]{TW}
H.~Triebel and H.~Winkelvo{\ss}.
\newblock Intrinsic atomic characterizations of function spaces on domains.
\newblock {\em Math. Z.}, 221(4):647--673, 1996.

\bibitem[Zy]{zygmund}
A.~Zygmund.
\newblock {\em Trigonometric series}.
\newblock Cambridge Univ. Press, Cambridge, 2nd edition, 1977.

\end{thebibliography}
\end{document}